\pdfoutput=1

\documentclass{amsart}

\usepackage[margin=1.2 in]{geometry}

\usepackage{verbatim}
\usepackage{xcolor}
\usepackage{tikz}
\usepackage{tikz-cd}
\usepackage{arydshln}
\usepackage{caption}
\usepackage{subcaption}
\usepackage{graphicx}
\usepackage[hyphens]{url}
\usepackage{amsmath}
\usepackage{bbm}
\usepackage{hyperref}
\usepackage{mathrsfs}

\newtheorem{theorem}{Theorem}[section]
\newtheorem{proposition}[theorem]{Proposition}
\newtheorem{lemma}[theorem]{Lemma}

\theoremstyle{definition}

\theoremstyle{remark}

\numberwithin{equation}{section}

\newcommand{\SExt}{\smash{\sqrt{\mathrm{Ext}}}}
\newcommand{\SExtn}{\sqrt{\mathrm{Ext}}}
\newcommand{\Li}{\mathrm{Li}}
\newcommand{\mcg}{\mathrm{Mod}_g}
\newcommand{\mc}{\mathbf{g}}
\renewcommand{\tt}{\mathcal{T}_g}

\newcommand{\qut}{\mathcal{Q}^1\mathcal{T}_g}

\newcommand{\qum}{\mathcal{Q}^1\mathcal{M}_g}

\newcommand{\mf}{\mathcal{MF}_g}
\newcommand{\pmf}{\mathcal{PMF}_g}

\newcommand{\IT}{\Sigma}

\newcommand{\M}{\mathrm{M}}

\newcounter{count}

\newcounter{counterk} 
 
\newcounter{counterl} 
 
\newcounter{counterc} 
 
\newcounter{countercc} 
 
\newcounter{countere} 
 
\newcounter{counterd} 
 
\newcounter{countern} 
 
\newcounter{counterr}
 
\begin{document}

\title[Effective mapping class group dynamics III]{Effective mapping class group dynamics III:\\ Counting filling closed curves on surfaces}

\author{Francisco Arana--Herrera}
\address{Department of Mathematics, Stanford University, 450 Jane Stanford Way, Stanford, CA 94305, USA.}
\email{farana@stanford.edu}

\begin{abstract}
	We prove a quantitative estimate with a power saving error term for the number of filling closed geodesics of a given topological type and length $\leq L$ on an arbitrary closed, orientable, negatively curved surface. More generally, we prove estimates of the same kind for the number of free homotopy classes of filling closed curves of a given topological type on a closed, orientable surface whose geometric intersection number with respect to a given filling geodesic current is $\leq L$. The proofs rely on recent progress made on the study of the effective dynamics of the mapping class group on Teichmüller space and the space of closed curves of a closed, orientable surface, and introduce a novel method for addressing counting problems of mapping class group orbits that naturally yields power saving error terms. This method is also applied to study counting problems of mapping class group orbits of Teichmüller space with respect to Thurston's asymmetric metric.
\end{abstract}

\maketitle


\thispagestyle{empty}

\tableofcontents

\section{Introduction}

Counting problems for closed geodesics on hyperbolic surfaces have been extensively studied since the 1950s. Huber's prime geodesic theorem \cite{Hu56} can be considered as one of the first major breakthroughs in this subject. According to this theorem, the number $p(X,L)$ of primitive closed geodesics of length $\leq L$ on a closed, orientable hyperbolic surface $X$ admits the following estimate,
\begin{equation}
\label{eq:fil_count}
p(X,L) = \Li\left(e^L\right)  + O_X\left(e^{(1-\kappa) L}\right),
\end{equation}
where the gap $\kappa = \kappa(X) > 0$ depends only on the smallest non-zero eigenvalue of the Laplacian of $X$, and where $\Li \colon [2,+\infty) \to \mathbf{R}$ is the Eulerian logarithmic integral
\[
\Li(x) := \int_2^x \frac{dt}{\log t}.
\]
Let us highlight the fact that $\Li(x)$ is asymptotic to $x/\log(x)$ as $x \to +\infty$. The backbone of Huber's proof, and of many of the subsequent improvements and generalizations due to several authors \cite{He76,Ra78,Sa80}, is Selberg's famous trace formula \cite{Se56}. A proof of an estimate analogous to (\ref{eq:fil_count}) for arbitrary compact, negatively curved Riemannian manifolds was given by Margulis in his thesis \cite{Mar04} using dynamical and geometric arguments.

Neither Huber's nor Margulis's methods can be used to prove an estimate for the number $s(X,L)$ of simple closed geodesics of length $\leq L$ on a closed, orientable hyperbolic surface $X$. This problem, which garnered great interest towards the end of the last century \cite{BS85,MR95a,MR95b,Ri01}, would witness major developments through the work of Mirzakhani. In her thesis \cite{Mir04,Mir08b}, Mirzakhani showed that the following asymptotic estimate holds as $L \to +\infty$,
\begin{equation}
\label{eq:simp_asymp}
s(X,L) \sim n(X) \cdot L^{6g-6},
\end{equation}
where $n(X) > 0$ is a constant depending on the geometry of $X$, and where $g \geq 2$ is the genus of $X$. An important driving force behind Mirzakhani's proof is the ergodicity of the action of the mapping class group on the space of singular measured foliations, a result proved by Masur in \cite{Mas85}. 

The methods introduced in Mirzakhani's thesis did not provide an error term. It would actually take more than 15 years and several developments in Teichmüller dynamics for Eskin, Mirzakhani, and Mohammadi \cite{EMM19} to show that
\begin{equation}
\label{eq:simp_count}
s(X,L) = n(X) \cdot L^{6g-6} + O_X\left(L^{6g-6-\kappa}\right),
\end{equation}
where the gap $\kappa = \kappa(g) > 0$ depends only on the genus of $X$. The main source of effective estimates in the proof of (\ref{eq:simp_count}) is the exponential mixing property of the Teichmüller geodesic flow, a result proved by Avila and Resende \cite{AR12} building on previous work of Avila, Gouëzel, and Yoccoz \cite{AGY06}.

Both the methods introduced in Mirzakhani's thesis and those introduced in her later work with Eskin and Mohammadi can be used to study more refined counting problems of closed geodesics on hyperbolic surfaces. Two closed curves on homeomorphic surfaces are said to have the same topological type if there exists a homeomorphism between the surfaces identifying their free homotopy classes. Given a closed, orientable hyperbolic surface $X$ of genus $g \geq 2$, a closed curve $\beta$ on $X$, and $L > 0$, denote by $c(X,\beta,L)$ the number of closed geodesic on $X$ of the same topological type as $\beta$ and length $\leq L$. In \cite{EMM19}, Eskin, Mirzakhani, and Mohammadi showed that if $\beta$ is simple then
\begin{equation}
\label{eq:simp_type_count}
c(X,\beta,L) = n(X,\beta) \cdot L^{6g-6} + O_{X,\beta}\left(L^{6g-6-\kappa}\right),
\end{equation}
where $n(X,\beta) > 0$ is a constant depending on the geometry of $X$ and the topological type of $\beta$, and where the gap $\kappa = \kappa(g) > 0$ depends only on the genus of $X$. Let us highlight the fact that the proof of this estimate makes crucial use of the assumption that the closed curve $\beta$ is simple.

As the estimates in (\ref{eq:fil_count}), (\ref{eq:simp_asymp}), and (\ref{eq:simp_count}) show, simple closed geodesics account for just a tiny fraction of all primitive closed geodesics of a closed, orientable hyperbolic surface. Furthermore, primitive closed geodesics are generically filling, i.e., all but a quantitatively small number of primitive closed geodesics of a closed, orientable hyperbolic surface cut the surface into discs.

Counting problems for closed geodesics of non-simple topological types had been previously studied by Mirzakhani. A closed curve on a closed, orientable surface is said to be filling if it intersects every homotopically non-trivial closed curve. In \cite{Mir16}, Mirzakhani showed that if $X$ is a closed, orientable hyperbolic surface of genus $g \geq 2$ and $\beta$ is a filling closed curve on $X$, then, asymptotically as $L \to +\infty$,
\begin{equation}
\label{eq:closed_asymp}
c(X,\beta,L) \sim n(X,\beta) \cdot L^{6g-6},
\end{equation}
where $n(X,\beta) > 0$ is a constant depending on the geometry of $X$ and the topological type of $\beta$. As in her thesis, the proof of this estimate is also based on dynamical arguments, but the key player in this case is the ergodicity of the earthquake flow, a result proved earlier by herself in \cite{Mir08a}. An asymptotic estimate analogous to (\ref{eq:closed_asymp}) for closed geodesics of any topological type was later proved by Erlandsson and Souto \cite{ES19} using original arguments introduced in their earlier work \cite{ES16}. Neither Mirzhakani's methods nor the methods of Erlandsson and Souto provide an error term.

The problem of proving a quantitative estimate with a power saving error term for the counting function $c(X,\beta,L)$ in the case where $\beta$ is non-simple has since remained open. This problem was recently advertised by Wright \cite[Problem 18.2]{Wri19}. In this paper we solve this problem for most topological types of closed curves by showing that a quantitative estimate analogous to (\ref{eq:simp_type_count}) holds in the case where $\beta$ is filling. More generally, we prove estimates of the same kind for the number of free homotopy classes of filling closed curves of a given topological type on a closed, orientable surface whose geometric intersection number with respect to a given filling geodesic current is $\leq L$.

In this paper we introduce a novel method for addressing counting problems of this nature which naturally yields power saving error terms. We refer to this method as the tracking method. This method relies on recent progress made in the prequels \cite{Ara20b} and \cite{Ara20c} on the study of the effective dynamics of the mapping class group on Teichmüller space and the space of closed curves of a closed, orientable surface. These recent developments in turn rely on the exponential mixing rate, the hyperbolicity, and the renormalization dynamics of the Teichmüller geodesic flow as their main driving forces. The tracking method can thus be interpreted as a new incarnation of these fundamental ideas in the context of counting problems for closed curves on surfaces.

This method can also be used to study counting problems of mapping class group orbits in other settings. For example, we prove a quantitative estimate with a power saving error term for the number of points in a mapping class group orbit of Teichmüller space that lie within a Thurston asymmetric metric ball of given center and large radius. This estimate effectivizes an asymptotic counting result of Rafi and Souto proved in \cite{RS19}.

\subsection*{Statement of the main theorem.} For the rest of this section we fix an integer $g \geq 2$ and a connected, oriented, closed surface $S_g$ of genus $g$. We refer the reader to \S 2 for a detailed treatment of the technical concepts and terminology that will be used throughout the rest of this section. In the following discussion we will not distinguish between closed curves on $S_g$ and their free homotopy classes. In addition, all metrics considered on $S_g$ will be Riemannian. 

Denote by $\mathcal{C}_g$ the space of geodesic currents on $S_g$. Closed curves and negatively curved metrics on $S_g$ embed naturally into $\mathcal{C}_g$. We refer to the geodesic current $h \in \mathcal{C}_g$ corresponding to a negatively curved metric $h$ on $S_g$ as its Liouville current. Denote by $i(\cdot,\cdot)$ the geometric intersection number pairing on $\mathcal{C}_g$. If $h$ is a negatively curved metric on $S_g$ and $\beta$ is a closed curve on $S_g$ then $i(h,\beta)$ is the length of the unique geodesic representative of $\beta$ with respect to $h$. A geodesic current $\alpha$ on $S_g$ is said to be filling if $i(\alpha,\beta) > 0$ for every non-zero $\beta \in \mathcal{C}_g$.  The Liouville current of any negatively curved metric on $S_g$ is filling. Denote by $\mcg$ the mapping class group of $S_g$. This group acts naturally on $\mathcal{C}_g$ preserving the geometric intersection number pairing.

Let $\alpha \in \mathcal{C}_g$ be a filling geodesic current on $S_g$ and $\beta$ be a closed curve on $S_g$. For every $L > 0$ consider the counting function
\[
c(\alpha,\beta,L) := \#\{ \gamma \in \mcg \cdot \beta  \ | \  i(\alpha,\gamma) \leq L \}.
\]
Notice that if $h \in \mathcal{C}_g$ is the Liouville current of a negatively curved metric $h$ on $S_g$ then $c(h,\beta,L)$ is the number of closed geodesics on $S_g$ with respect to $h$ of the same topological type as $\beta$ and length $\leq L$. Thus, this definition recovers the counting function $c(X,\beta,L)$ introduced above for closed, orientable hyperbolic surfaces $X$ and closed curves $\beta$ as a particular case. The following effective estimate for the counting function $c(\alpha,\beta,L)$ in the case where $\beta$ is filling is the main result of this paper.

\begin{theorem}
	\label{theo:main}
	There exists $\kappa = \kappa(g) > 0$ such that for every filling geodesic current $\alpha \in \mathcal{C}_g$ and every filling closed curve $\beta$ on $S_g$ there exists a constant $n(\alpha,\beta) > 0$ such that for every $L > 0$,
	\[
	c(\alpha,\beta,L) = n(\alpha,\beta) \cdot L^{6g-6} + O_{\alpha,\beta}\left(L^{6g-6-\kappa}\right).
	\]
\end{theorem}

A version of Theorem \ref{theo:main} with a more explicit description of the leading term and a more precise control of the implicit constant will be introduced in \S5 as Theorem \ref{theo:main_strong}. A particular application of Theorem \ref{theo:main} solves the open problem of proving a quantitative estimate with a power saving error term for the number $c(X,\beta,L)$ of closed geodesics of the same topological type as $\beta$  and length $\leq L$ on a  closed, orientable hyperbolic surface $X$ in the generic case where $\beta$ is filling.

A stronger version of Theorem \ref{theo:main} for counting functions of non-connected filling closed multi-curves on closed, orientable surfaces will be introduced in \S 5 as Theorem \ref{theo:main_multi}. A particular application of Theorem \ref{theo:main_multi} effectivizes asymptotic estimates of the author \cite{A20} for counting functions of filling closed multi-geodesics on closed, orientable hyperbolic surfaces that keep track of the lengths of the individual components of the multi-geodesics.

\subsection*{Counting with respect to Thurston's asymmetric metric.} The techniques introduced in the proof of Theorem \ref{theo:main} can also be applied to study counting problems of mapping class group orbits in other settings. Denote by $\mathcal{T}_g$ the Teichmüller space of marked hyperbolic structures on $S_g$. Thurston's asymmetric metric quantifies the minimal Lipschitz constant among Lipschitz maps between marked hyperbolic structures on $S_g$. More precisely, for every $X,Y \in \mathcal{T}_g$,
\[
d_\mathrm{Thu}(X,Y) := \log\left( \inf_{f \colon X \to Y} \mathrm{Lip}(f)\right),
\]
where the infimum runs over all Lipschitz maps $f \colon X \to Y$ in the homotopy class given by the markings of $X$ and $Y$, and where $\mathrm{Lip}(f)$ denotes the Lipschitz constant of such maps. Denote by $A_R(X) \subseteq \mathcal{T}_g$ the closed ball of radius $R > 0$ centered at $X \in \mathcal{T}_g$ with respect to Thurston's asymmetric metric.

The mapping class group $\mcg$ acts naturally on $\mathcal{T}_g$ by changing the markings. This action preserves Thurston's asymmetric metric. In \cite{RS19}, Rafi and Souto proved an asymptotic estimate for the number of points in a mapping class group orbit of Teichmüller space that lie within a Thurston asymmetric metric ball of given center and large radius. More precisely, Rafi and Souto showed that for every $X,Y \in \mathcal{T}_g$ the following asymptotic estimate holds as $R \to +\infty$,
\begin{equation}
\label{eq:asymp_thu}
\#\left(\mcg \cdot Y \cap A_R(X)\right) \sim C(X,Y) \cdot e^{(6g-6)R},
\end{equation}
where $C(X,Y) > 0$ is a constant depending on the geometry of $X$ and $Y$. Using the techniques introduced in the proof of Theorem \ref{theo:main} we prove the following effectivization of (\ref{eq:asymp_thu}). 

\begin{theorem}
	\label{theo:main_3}
	There exists a constant $\kappa = \kappa(g) > 0$ such that for every pair of marked hyperbolic structures $X,Y \in \mathcal{T}_g$ there exists a constant $C(X,Y) > 0$ such that for every $R > 0$,
	\[
	\#\left(\mcg \cdot Y \cap A_R(X)\right) = C(X,Y) \cdot e^{(6g-6)R} + O_{X,Y}\left(e^{(6g-6-\kappa)R} \right).
	\]
\end{theorem}

A version of Theorem \ref{theo:main_3} with a more explicit description of the leading term and a more precise control of the implicit constant will be introduced in \S6 as Theorem \ref{theo:main_3_strong}.

\subsection*{The tracking method.} The proof of Theorem \ref{theo:main} introduces a novel method for addressing counting problems of mapping class group orbits of filling, closed curves on closed, orientable surfaces which naturally yields power saving error terms. We refer to this method as the tracking method. At the heart of this method is the tracking principle proved in the prequel \cite{Ara20c}. According to this principle, the action of the mapping class group on the space of closed curves of a closed, orientable surface effectively tracks the corresponding action on Teichmüller space in the following sense: for all but quantitatively few mapping classes, the information of how a mapping class moves a given point of Teichmüller space determines, up to a power saving error term, how it changes the geometric intersection numbers of a given closed curve with respect to arbitrary geodesic currents. See Theorem \ref{theo:track} for a precise statement of this principle. The main driving force behind this principle is the renormalization dynamics of the Teichmüller geodesic flow, in particular, the functional analytic approach to renormalization introduced by Forni \cite{F02} and developed by Athreya and Forni \cite{AF08}.

A careful application of the tracking principle reduces the original counting problem for mapping class group orbits of filling closed curves on closed, orientable surfaces to a counting problem for mapping class group orbits of Teichmüller space. This represents a significant advantage because the tools available for studying such counting problems on Teichmüller space are much more robust than those available in the original setting. A careful application of the effective bisector counting estimates for mapping class group orbits of Teichmüller space proved in the prequel \cite{Ara20b} leads to a solution of the reduced counting problem. See Theorem \ref{theo:bisect_count} for a precise statement of the estimates used in the proof. These estimates are inspired by ideas introduced by Margulis in his thesis \cite{Mar04}, build on previous work of Athreya, Bufetov, Eskin, and Mirzakhani \cite{ABEM12}, and rely on the exponential mixing property of the Teichmüller geodesic flow \cite{AGY06,AR12} in a crucial way. 

To ensure the estimates considered throughout the proof yield a power saving error term in the end, we study the regularity of certain functions of interest on the space of singular measured foliations of a closed, orientable surface. More specifically, we show that the geometric intersection number with respect to a given geodesic current and the extremal length with respect to a given marked complex structure define Lipschitz functions on the space of singular measured foliations when parametrized using Dehn-Thurston coordinates. See Theorems \ref{theo:curr_lip} and \ref{theo:ext_lip} for precise statements.

\subsection*{Non-filling closed curves.} The importance in the proof of Theorem \ref{theo:main} of the assumption that the closed curve $\beta$ is filling is twofold. On one hand, this assumption guarantees that the stabilizer of $\beta$ with respect to the mapping class group action is finite, so counting closed curves in the mapping class group orbit of $\beta$ is equivalent to counting mapping classes in terms of how they act on $\beta$. On the other hand, this assumption ensures that $\beta$ gets streched at an exponential rate along any Teichmüller geodesic flow orbit, so the tracking principle described above can be applied directly. The author believes the techniques introduced in this paper should serve as a good starting point for a proof of a version of Theorem \ref{theo:main} that holds in the general case where the closed curve $\beta$ is not filling.

\subsection*{Finite index subgroups of the mapping class group.} The main results of this paper, Theorems \ref{theo:main} and \ref{theo:main_3}, and their stronger counterparts, Theorems \ref{theo:main_strong}, \ref{theo:main_multi}, and \ref{theo:main_3_strong}, also hold when the corresponding counting functions are defined using an arbitrary finite index subgroup of the mapping class group. This follows from the fact that the bisector counting estimates for mapping class group orbits of Teichmüller space used in the proofs of these results also hold for arbitrary finite index subgroups of the mapping class group, as explained in the prequel \cite{Ara20b}.

\subsection*{Organization of the paper.} In \S 2 we cover the background material and set up the notation that will be used throughout the rest of this paper. In \S 3 we discuss the main results of the prequels \cite{Ara20b} and \cite{Ara20c} concerning the effective dynamics of the mapping class group on Teichmüller space and the space of closed curves of a closed, orientable surface. In \S 4 we study the regularity of geometric intersection numbers and extremal lengths on the space of singular measured foliations of a closed, orientable surface. In \S 5 we use the tracking method to prove Theorems \ref{theo:main}. In \S 6 we discuss how to apply the tracking method to prove Theorem \ref{theo:main_3}.

\subsection*{Acknowledgments.} The author is very grateful to Alex Wright and Steve Kerckhoff for their invaluable advice, patience, and encouragement. The author would especially like to thank Alex Eskin for very influential conversations during the early stages of this project. The author would also like to thank Jayadev Athreya, Jenya Sapir, and Benjamin Dozier for very enlightening conversations. This work got started while the author was participating in the ``Dynamics: Topology and Numbers" trimester program at the Hausdorff Research Institute for Mathematics (HIM). The author is very grateful for the HIM's hospitality and for the hard work of the organizers of the trimester program.

\section{Background material}

\subsection*{Outline of this section.} In this section we cover the background material and set up the notation that will be used throughout the rest of this paper. We begin with a quick overview of the Teichmüller metric and its relation to the Teichmüller geodesic flow. We then cover some aspects of the theory of singular measured foliations and their different coordinate systems. We end this section with a brief discussion of extremal lengths and geodesic currents.

\subsection*{The Teichmüller metric.} For the rest of this section we fix an integer $g \geq 2$ and a connected, oriented, closed surface $S_g$ of genus $g$. Denote by $\mathcal{T}_g$ the Teichmüller space of marked complex structures on $S_g$ and by $\mcg$ the mapping class group of $S_g$. As explained in \S 1, the proof of Theorem \ref{theo:main}, the main result of this paper, relies on the dynamics of the marking changing action of the mapping class group on Teichmüller space in a crucial way. In \S 3 we provide a quantitative description of these dynamics in terms of the Teichmüller metric. See Theorems \ref{theo:teich_count} and \ref{theo:bisect_count}. 

The Teichmüller metric $d_\mathcal{T}$ on $\mathcal{T}_g$ quantifies the minimal dilation among quasiconformal maps between marked complex structures on $S_g$. More precisely, for every $X,Y \in \mathcal{T}_g$,
\[
d_\mathcal{T}(X,Y) := \log\left( \inf_{f \colon X \to Y} K(f)\right),
\]
where the infimum runs over all quasiconformal maps $f \colon X \to Y$ in the homotopy class given by the markings of $X$ and $Y$, and where $K(f)$ denotes the dilation of such maps. See \cite[Chapter 11]{FM11} for a more detailed definition. The action of the mapping class group on Teichmüller space preserves the Teichmüller metric. The Teichmüller metric is complete and its geodesics can be described explicitly in terms of a natural flow on the unit cotangent bundle of Teichmüller space, as we now explain.

\subsection*{The Teichmüller geodesic flow.} Let $X$ be a closed Riemann surface and $K$ be its canonical bundle. A quadratic differential $q$ on $X$ is a holomorphic section of the symmetric square $K \vee K$. In local coordinates, $q = f(z) \thinspace dz^2$ for some holomorphic function $f(z)$. If $X$ has genus $g$, the number of zeroes of $q$ counted with multiplicity is $4g-4$. The zeroes of $q$ are also called singularities. We sometimes denote quadratic differentials by $(X,q)$ to keep track of the Riemann surface they are defined on.

A half-translation structure on a surface $S$ is an atlas of charts to $\mathbf{C}$ on the complement of a finite set of points $\Sigma \subseteq S$ whose transition functions are of the form $z \mapsto \pm z + c$ with $c \in \mathbf{C}$. Every quadratic differential $q$ gives rise to a half-translation structure on the Riemann surface it is defined on by considering local coordinates on the complement of the zeroes of $q$ for which $q = dz^2$. Viceversa, every half-translation structure induces a quadratic differential on its underlying surface by pulling back the differential $dz^2$ on the corresponding charts. The area of a quadratic differential is the total Euclidean area of the corresponding half-translation structure.

Pulling back the measured foliations corresponding to the $1$-forms $dx$ and $dy$ on $\mathbf{C}$ using the charts of a half-translation structure induces a pair of singular measured foliation on the underlying surface. For the half translation structure induced by a quadratic differential $q$, we denote these singular measured foliations by $\Re(q)$ and $\Im(q)$, and refer to them as the vertical and horizontal foliations of $q$. 

The group $\mathrm{SL}(2,\mathbf{R})$ acts naturally on half-translation structures by postcomposing the corresponding charts with the linear action of $\mathrm{SL}(2,\mathbf{R})$ on $\mathbf{C} =\mathbf{R}^2$. In particular, $\mathrm{SL}(2,\mathbf{R})$ acts naturally on quadratic differentials preserving their area. The Teichmüller geodesic flow on quadratic differentials corresponds to the action of the diagonal subgroup $\{a_t\}_{t \in \mathbf{R}} \subseteq \mathrm{SL}(2,\mathbf{R})$ given by
\begin{equation*}
a_t := \left( 
\begin{array}{c c}
e^t & 0 \\
0 & e^{-t}
\end{array} 
\right).
\end{equation*}

Denote by $\qut$ the Teichmüller space of marked unit area quadratic differentials on $S_g$. This space can be canonically identified with the unit cotangent bundle of $\mathcal{T}_g$ by considering the natural projection $\pi \colon \qut \to \mathcal{T}_g$ which maps every marked unit area quadratic differential on $S_g$ to its underlying marked complex structure. The natural $\mathrm{SL}(2,\mathbf{R})$ action on quadratic differentials induces a corresponding action on $\qut$. The unit speed geodesics of the Teichmüller metric are precisely the projection to $\mathcal{T}_g$ of the unit speed orbits of the Teichmüller geodesic flow on $\qut$.

\subsection*{The Masur-Veech measure.} Via its canonical identification with the unit cotangent bundle of $\mathcal{T}_g$, the space $\qut$ inherits a natural measure $\mu_{\mathrm{mv}}$ known as the Masur-Veech measure. We normalize this measure in the same way as \cite[\S 2]{Ara20b}. The natural properly discontinuous marking changing action of the mapping class group on $\qut$ preserves the Masur-Veech measure. Denote by $\widehat{\mu}_{\mathrm{mv}}$ local pushforward of $\mu$ to the quotient $\qum := \qut/\mcg$. Independent works of Masur \cite{Mas82} and Veech \cite{Ve82} guarantee this measure is finite. Throughout this paper we consider the constant
\begin{equation}
\label{eq:bg}
b_g := (6g-6) \cdot \widehat{\mu}_{\mathrm{mv}}\left(\qum\right).
\end{equation}

\subsection*{The boundary at infinity of Teichmüller space.} A crucial step in the proof of Theorem \ref{theo:main}, the main result of this paper, consists in appropriately localizing certain counting functions of mapping class group orbits of Teichmüller space in terms of information on the boundary at infinity. Denote by $\mf$ the space of singular measured foliations on $S_g$. Weighted homotopy classes of simple closed curves on $S_g$ embed densely into $\mf$ \cite[Proposition 6.18]{FLP12}. The space $\mf$ can be endowed with a natural $\mathbf{R}^+$ action which scales transverse measures. Denote by $\pmf$ the space of projective singular measured foliations on $S_g$. This space is compact \cite[Theorem 6.15]{FLP12}. Given a singular measured foliation $\lambda \in \mf$ denote by $[\lambda] \in \pmf$ its projective class. 

The space $\pmf$ can be interpreted as the boundary at infinity of $\mathcal{T}_g$ in the following way. The projective classes $[\Re(q)], [\Im(q)] \in \pmf$ of the horizontal and vertical foliations $\Re(q), \Im(q) \in \mf$ of a marked unit area quadratic differential $q \in \qut$ correspond to the forwards and backwards endpoints on the boundary at infinity of the Teichmüller geodesic generated by $q$. Furthermore, by work of Hubbard and Masur \cite{HM79}, every pair of projective singular measured foliations satisfying a natural transversality condition corresponds to a unique $q \in \qut$ in this way. The same work of Hubbard and Masur also shows that for every $X \in \mathcal{T}_g$ and every projective singular measured foliation $[\lambda] \in \pmf$ there exists a unique $q \in \pi^{-1}(X)$ such that $[\Re(q)] = [\lambda]$.

Denote by $\Delta \subseteq \mathcal{T}_g \times \mathcal{T}_g$ the diagonal of $\mathcal{T}_g \times \mathcal{T}_g$ and by $S(X) := \pi^{-1}(X)$ the fiber of the projection $\pi \colon \qut \to \tt$ above $X \in \mathcal{T}_g$. Consider the maps $q_s,q_e \colon \mathcal{T}_g \times \mathcal{T}_g - \Delta \to \mathcal{Q}^1\mathcal{T}_g$ which to every pair of distinct points $X \neq Y \in \mathcal{T}_g$ assign the marked quadratic differentials $q_s(X,Y) \in S(X)$ and $q_e(X,Y) \in S(Y)$ corresponding to the cotangent directions at $X$ and $Y$ of the unique Teichmüller geodesic segment from $X$ to $Y$. For $X \neq Y \in \mathcal{T}_g$, if $r := d_\mathcal{T}(X,Y) > 0$, then $q_e(X,Y) = a_r.q_s(X,Y)$. The localization procedure used in the proof of Theorem \ref{theo:main} is performed in terms of sectors of Teichmüller space defined in the following way. Given $X \in \mathcal{T}_g$ and $U \subseteq \pmf$ consider the sector
\[
\mathrm{Sect}_U(X) := \{Y \in \mathcal{T}_g \setminus \{X\} \ | \ [\Re(q_s(X,Y))] \in U\}.
\]

\subsection*{Coordinate systems on the space of singular measured foliations.} It will be convenient for our purposes to consider two different coordinates systems on the space of singular measured foliations: train track coordinates and Dehn-Thurston coordinates. We now give a quick overview of the fundamental aspects of these coordinate systems. For more details we refer the reader to \cite{PH92}.

Let us begin by describing train track coordinates. Train track coordinates have the advantage of being quite flexible in their construction and appearing rather naturally in the study of related subjects. A train track $\tau$ on $S_g$ is an embedded $1$-complex satisfying the following conditions:
\begin{enumerate}
	\item Each edge of $\tau$ is a smooth path with a well defined tangent vector at each endpoint. All edges at a given vertex are tangent.
	\item For each component $R$ of $S_g \setminus \tau$, the double of $R$ along the smooth part of its boundary $\partial R$ has negative Euler characteristic.
\end{enumerate}
The vertices of a train track where three or more edges meet are called switches. Without loss of generality we assume all the vertices of the train tracks considered are switches. By considering the inward pointing tangent vectors of the edges incident to a switch, one can divide these edges into incoming and outgoing edges. A train track $\tau$ on $S_g$ is said to be maximal if all the components of $S_g \setminus \tau$ are trigons, i.e., interiors of discs with three non-smooth points on their boundaries.

A singular measured foliation $\lambda \in \mf$ is said to be carried by a train track $\tau$ on $S_g$ if it can be obtained by collapsing the complementary regions in $S_g$ of a measured foliation of a tubular neighborhood of $\tau$ whose leaves run parallel to the edges of $\tau$. In this situation, the invariant transverse measure of $\lambda$ corresponds to a counting measure $v$ on the edges of $\tau$ satisfying the switch conditions: at every switch of $\tau$ the sum of the measures of the incoming edges equals the sum of the measures of the outgoing edges. Every $\lambda \in \mf$ is carried by some maximal train track $\tau$ on $S_g$. 

Let $\tau$ be a maximal train track on $S_g$. Denote by $U(\tau) \subseteq \mf$ the cone of singular measured foliations on $S_g$ carried by $\tau$ and by $V(\tau) \subseteq \smash{(\mathbf{R}_{\geq0})^{18g-18}}$ the $6g-6$ dimensional cone of non-negative counting measures on the edges of $\tau$ satisfying the switch conditions. These cones can be identified through a natural $\mathbf{R}^+$-equivariant bijection $\Phi_\tau \colon U(\tau) \to V(\tau)$ we refer to as the train track chart induced by $\tau$ on $\mf$. These charts give rise to coordinates on $\mf$ called train track coordinates. The transition maps of these coordinates are piecewise integral linear. Thus, these coordinates can be used to endow $\mf$ with a natural $6g-6$ dimensional piecewise integral linear structure. The Thurston measure $\nu_\mathrm{Thu}$ is the unique, up to scaling, piecewise integral linear measure on $\mf$. We normalize this measure in the same way as \cite[\S 2]{Ara20b}.

Dehn-Thurston coordinates parametrize $\mf$ in terms of intersection and twisting numbers with respect to the components of a pair of pants decomposition of $S_g$. Consider the piecewise linear manifold $\IT:= \mathbf{R}^2 / \langle-1\rangle$ endowed with the quotient Euclidean metric. The product $\smash{\IT^{3g-3}}$ is a piecewise linear manifold which we endow with the $L^2$ product metric. Every set of Dehn-Thurston coordinates provides an $\mathbf{R}^+$-equivariant homeomorphism $F \colon \mf \to \smash{\IT^{3g-3}}$. The change of coordinates maps between Dehn-Thurston coordinates and train track coordinates are piecewise linear. Although Dehn-Thurston coordinates are rather restricted in their construction they have the advantage of providing global piecewise linear parametrizations of $\mf$.

\subsection*{Extremal lengths.} An important aspect of the proof of Theorem \ref{theo:main}, the main result of this paper, is controling the regularity of the codimension-$1$ subset of $\mf$ of all horizontal foliations $\Re(q)$ of marked unit area quadratic differentials $q \in \qut$ that lie over a given marked complex structure $X \in \mathcal{T}_g$. The regularity of this subset is studied using the notion of extremal length.

Let $X \in \mathcal{T}_g$ be a marked complex structure on $S_g$. Fix a Riemannian metric $h$ on $S_g$ in the conformal class of $X$. Denote by $dA_h$ and $ds_h$ the area and length elements of $h$. Every metric on $S_g$ in the conformal class of $X$ can be obtained as a product $\rho h$ for some non-negative measurable function $\rho \colon S_g \to \mathbf{R}_{\geq 0}$. For any such function define the total area of the metric $\rho h$ as
\[
\mathrm{Area}(\rho h) := \int_X \rho^2 \thinspace dA_h.
\]
Let $\gamma$ be a simple closed curve on $S_g$. For every non-negative measurable function $\rho \colon S_g \to \mathbf{R}_{\geq 0}$ define
\[
\ell_{\rho h}(\gamma) := \inf_{\gamma' \sim \gamma} \int_{\gamma'} \rho \thinspace ds_h,
\]
where the infimum runs over all simple closed curves $\gamma'$ on $S_g$ homotopic to $\gamma$. The square root of the extremal length of $\gamma$ with respect to $X$ is defined as
\[
\SExt_X(\gamma) := \sup_{\rho} \frac{\ell_{\rho h}(\gamma)}{\sqrt{\mathrm{Area}(\rho h)}},
\]
where the supremum runs over all non-negative measurable functions $\rho \colon S_g \to \mathbf{R}_{\geq 0}$ such that $0 < \mathrm{Area}(\rho h)<+\infty$. By work of Kerckhoff \cite{Ker80}, there exists a unique continuous, homogeneous extension of $\SExt_X$ to a function $\SExt_X \colon \mf \to \mathbf{R}$ on the space of singular measured foliations. Furthermore, this extension satisfies the following crucial identify for every $q \in \qut$,
\[
\SExt_{\pi(q)}(\Re(q)) = \mathrm{Area}(q) = 1.
\]

In the context of the study of counting problems for mapping class group orbits of Teichmüller space, Athreya, Bufetov, Eskin, and Mirzakhani \cite{ABEM12} introduced the Hubbard-Masur function $\Lambda \colon \mathcal{T}_g \to \mathbf{R}^+$ which to every marked complex structure $X \in \mathcal{T}_g$ assigns the positive value
\[
\Lambda(X) := \nu_{\mathrm{Thu}}\left(\{\lambda \in \mf \ | \ \SExt_X(\lambda) \leq 1 \} \right).
\]
Combining results of Dumas \cite{Du15} with Gardiner’s variational formula \cite{Ga84}, Mirzakhani showed the function $\Lambda \colon \mathcal{T}_g \to \mathbf{R}^+$ is constant \cite[Theorem 5.10]{Du15}. Following the convention in \cite{Ara20b}, we refer to the value of this function as the Hubbard-Masur constant and denote it by
\begin{equation}
\label{eq:hmc}
\Lambda_g := \Lambda(X), \ \forall X \in \mathcal{T}_g.
\end{equation}

\subsection*{Geodesic currents.} In \cite{Bon88}, Bonahon gave a unified treatment of several seemingly unrelated notions of length for closed curves on closed, orientable surfaces using the concept of geodesic currents. Theorem \ref{theo:main}, the main result of this paper, is stated using this general notion.

To define geodesic currents let us endow the surface $S_g$ with an auxiliary hyperbolic metric. The projective tangent bundle $PTS_g$ admits a $1$-dimensional foliation by lifts of geodesics on $S_g$. A geodesic current on $S_g$ is a Radon transverse measure of the geodesic foliation of $PTS_g$. Equivalently, a geodesic current on $S_g$ is a $\pi_1(S_g)$-invariant Radon measure on the space of unoriented geodesics of the universal cover of $S_g$. Endow the space of geodesic currents on $S_g$ with the weak-$\star$ topology. Different choices of auxiliary hyperbolic metrics on $S_g$ yield canonically identified spaces of geodesic currents \cite[Fact 1]{Bon88}. Denote the space of geodesic currents on $S_g$ by $\mathcal{C}_g$. This space supports a natural $\mathbf{R}^+$ scaling action and a natural $\mcg$ action \cite[\S 2]{RS19}. Denote by  $\mathcal{PC}_g$ the space of projective geodesic currents on $S_g$. This space is compact \cite[Corollary 5]{Bon88}.

Homotopy classes of weighted closed curves on $S_g$ embed into $\mathcal{C}_g$ by considering their geodesic representatives with respect to any auxiliary hyperbolic metric. By work of Bonahon \cite[Proposition 2]{Bon88}, this embedding is dense. Moreover, the geometric intersection number pairing for closed curves on $S_g$ extends in a unique way to a continuous, symmetric, bilinear pairing $i(\cdot,\cdot)$ on $\mathcal{C}_g$ \cite[Proposition 3]{Bon88}. This pairing is invariant with respect to the diagonal action of $\mcg$ on $\mathcal{C}_g \times \mathcal{C}_g$.

Many different metrics on $S_g$ embed into $\mathcal{C}_g$ in such a way that the geometric intersection number of the metric with any closed curve is equal to the length of the geodesic representatives of the closed curve with respect to the metric. We refer to the geodesic current corresponding to any such metric as its Liouville current. Examples of metrics admitting Liouville currents include:
\renewcommand{\labelitemi}{\textperiodcentered}
\begin{itemize}
	\item Hyperbolic metrics \cite{Bon88},
	\item Negatively curved Riemannian metrics \cite{O90},
	\item Negatively curved Riemannian metrics with cone singularities of angle $\geq 2\pi$ \cite{HP97},
	\item Singular flat metrics induced by quadratic differentials \cite{DLR10},
	\item Singular flat metrics with cone singularities of angle $\geq 2\pi$ \cite{BL18}.
\end{itemize}

A closed curve on $S_g$ is said to be filling if it intersects every homotopically non-trivial closed curve on $S_g$. A geodesic current $\alpha \in \mathcal{C}_g$ is said to be filling if $i(\alpha,\beta) > 0$ for every non-zero $\beta \in \mathcal{C}_g$. Denote by $\mathcal{C}_g^* \subseteq \mathcal{C}_g$ the open subset of filling geodesic currents on $S_g$. Relevant examples of filling geodesic currents include filling closed curves and the Liouville currents listed above. Singular measured foliations  on $S_g$ embed into $\mathcal{C}_g$ by considering their geodesic representatives with respect to any auxiliary hyperbolic metric \cite{Lev83}. Given a filling geodesic current $\alpha \in \mathcal{C}_g^*$ consider the positive constant
\begin{equation*}
c(\alpha) := \nu_{\mathrm{Thu}}\left(\{\lambda \in \mf \ | \ i(\alpha,\lambda) \leq 1\}\right).
\end{equation*}

\subsection*{Notation.} Let $A,B \in \mathbf{R}$ and $*$ be a set of parameters. We write $A \preceq_* B$ if there exists a constant $C= C(*) > 0$ depending only on $*$ such that $A \leq C \cdot B$. We write $A \asymp_* B$ if $A \preceq_* B $ and $B \preceq_* A$. We write $A = O_*(B)$ if there exists a constant $C = C(*) > 0$ depending only on $*$ such that $|A| \leq C \cdot B$. 

\section{Effective mapping class group dynamics}

\subsection*{Outline of this section.} In this section we discuss the main results of the prequels \cite{Ara20b} and \cite{Ara20c} needed to apply the tracking method sketched in \S 1 to prove Theorem \ref{theo:main}. We first discuss the effective estimates for counting functions of mapping class group orbits of Teichmüller space proved in the prequel \cite{Ara20b}. See Theorems \ref{theo:teich_count} and \ref{theo:bisect_count}. These estimates are the main counting tool used in the proof of Theorem \ref{theo:main}. We then discuss the main result of the prequel \cite{Ara20c}, which shows that the action of the mapping class group on the space of closed curves of a closed, orientable surface effectively tracks the corresponding action on Teichmüller space. See Theorem \ref{theo:track}. This result is the cornerstone of the tracking method. 

\subsection*{Counting mapping class group orbits of Teichmüller space.} For the rest of this section we fix an integer $g \geq 2$ and a connected, oriented, closed surface $S_g$ of genus $g$. Recall that $\tt$ denotes the Teichmüller space of marked complex structures on $S_g$ and that $d_\mathcal{T}$ denotes the Teichmüller metric on $\mathcal{T}_g$. Recall that $\mcg$ denotes the mapping class group of $S_g$, that this group acts properly discontinuously on $\tt$ by changing the markings, and that this action preserves the Teichmüller metric. For every pair of marked complex structures $X,Y \in \tt$ and every $R > 0$ consider the counting function
\[	
N(X,Y,R) := \#\{\mc \in \mcg \ | \ d_\mathcal{T}(X,\mc.Y)\}.
\]

Recall the definitions of the constants $b_g > 0$ and $\Lambda_g > 0$ introduced in (\ref{eq:bg}) and (\ref{eq:hmc}). In the prequel \cite{Ara20b}, building on previous work of Athreya, Bufetov, Eskin, and Mirzakhani \cite{ABEM12}, we prove the following effective estimate for the counting function $N(X,Y,R)$.
 
\begin{theorem} 
	\cite[Theorem 4.1]{Ara20b}
 	\label{theo:teich_count}
 	There exists a constant $\kappa = \kappa(g) > 0$ such that for every compact subset $\mathcal{K} \subseteq \tt$, every $X,Y \in \mathcal{K}$, and every $R>0$,
 	\[
 	N(X,Y,R)=  \frac{\Lambda_g^2}{b_g}\cdot e^{(6g-6)R} + O_\mathcal{K}\left(e^{(6g-6-\kappa)R}\right).
 	\]
\end{theorem}

To apply the tracking method sketched in \S 1 to prove Theorem \ref{theo:main} we need effective estimates for counting functions of mapping class group orbits of Teichmüller space more specialized than $N(X,Y,R)$. To define these counting functions precisely let us first review some of the notation introduced in previous sections. Recall that $\qut$ denotes the Teichmüller space of marked unit area quadratic differentials on $S_g$, that $\pi \colon \qut \to \tt$ denotes the natural projection to $\mathcal{T}_g$, and that $S(X) := \pi^{-1}(X)$ for every $X \in \mathcal{T}_g$. Recall that $\Delta \subseteq \mathcal{T}_g \times \mathcal{T}_g$ denotes the diagonal of $\mathcal{T}_g$ and that $q_s,q_e \colon \mathcal{T}_g \times \mathcal{T}_g - \Delta \to \mathcal{Q}^1\mathcal{T}_g$ denote the maps which to every pair of distinct points $X \neq Y \in \mathcal{T}_g$ assign the marked quadratic differentials $q_s(X,Y) \in S(X)$ and $q_e(X,Y) \in S(Y)$ corresponding to the cotangent directions at $X$ and $Y$ of the unique Teichmüller geodesic segment from $X$ to $Y$. 

Recall that $\mf$ denotes the space of singular measured foliations on $S_g$ and that $\mathbf{R}^+$ acts naturally on this space by scaling the transverse measures. Recall that $\pmf$ denotes the space of projective singular measured foliations on $S_g$ and that $[\lambda] \in \pmf$ denotes the equivalence class of $\lambda \in \mf$. Consider the maps $\Re,\Im \colon \qut \to \mf$ which to every $q \in \qut$ assign its vertical and horizontal foliations $\Re(q), \Im(q) \in \mf$. Given $X \in \mathcal{T}_g$ and $U \subseteq \pmf$ consider the sector
\[
\mathrm{Sect}_U(X) := \{Y \in \mathcal{T}_g \setminus \{X\} \ | \ [\Re(q_s(X,Y))] \in U\}.
\]
For every $X,Y \in \mathcal{T}_g$, every $U,V \subseteq \pmf$, and every $R > 0$, consider the bisector counting function
\begin{equation}
\label{eq:N_count}
N(X,Y,U,V,R) = \# \left\lbrace
\begin{array}{l | l}
\mc \in \mcg \ & \ d_\mathcal{T}(X,\mc.Y) \leq R, \\ 
& \ \mc.Y \in \mathrm{Sect}_U(X), \\
& \ \mc^{-1}.X \in \mathrm{Sect}_V(Y).
\end{array}
\right\rbrace.
\end{equation}

To give an effective estimate for this counting function we restrict our attention to a particular class of subsets $U,V \subseteq \pmf$ we now introduce. Recall that $\Sigma$ denotes the piecewise linear manifold $\Sigma:=\mathbf{R}^2 / \langle-1\rangle$ endowed with the quotient Euclidean metric. Recall that the product $\IT^{3g-3}$ is a piecewise linear manifold which we endow with the $L^2$ product metric. The projectivization $P\Sigma^{3g-3}$ of $\Sigma^{3g-3}$ can be canonically identified with the subset $\smash{\Sigma_u^{3g-3}} \subseteq \smash{\Sigma^{3g-3}}$ of points in $\Sigma^{3g-3}$ of unit $L^1$-norm through a homeomorphism we denote by $u \colon \pmf \to \smash{\Sigma_u^{3g-3}}$. We endow the subset $\smash{\Sigma_u^{3g-3}} \subseteq \Sigma^{3g-3}$ with the induced Euclidean metric. The subset $\smash{\Sigma_u^{3g-3}} \subseteq \Sigma^{3g-3}$ can be represented as a union of $2^{3g-3}$ closed affine simplices of dimension $6g-7$ we refer to as the facets of $\smash{\Sigma_u^{3g-3}}$.

Recall that any set of Dehn-Thurston coodinates provides a homeomorphism $F \colon \mf \to \IT^{3g-3}$ equivariant with respect to the natural $\mathbf{R}^+$ scaling actions on $\mf$ and $\IT^{3g-3}$. We denote the projectivization of any such homeomorphism by $PF \colon \pmf \to P\Sigma^{3g-3}$. A measurable subset $U \subseteq \pmf$ is said to be an $F$-simplex if there exists an open affine simplex $W \subseteq \smash{\Sigma_u^{3g-3}}$ of dimension $6g-7$ contained in an $F$-facet of $\smash{\Sigma_u^{3g-3}}$ such that $W \subseteq u \circ PF(U) \subseteq \smash{\overline{W}}$.

Recall that $\SExt_X(\lambda)$ denotes the square root of the extremal length of a singular measured foliation $\lambda \in \mf$ with respect to a marked complex structure $X \in \mathcal{T}_g$. Recall that $\nu_\mathrm{Thu}$ denotes the Thurston measure on $\mf$. For every $X \in \mathcal{T}_g$ denote by $\nu_X$ the measure on $\pmf$ which to every measurable subset $U \subseteq \pmf$ assigns the value
\[
\nu_X(U) = \nu_{\mathrm{Thu}}\left(\left\lbrace\lambda \in \mf \ | \ \SExt_X(\lambda) \leq 1, \ [\lambda] \in U\right\rbrace\right).
\]

In the prequel \cite{Ara20b} we prove the following effective estimate for the bisector counting function $N(X,Y,U,V,R)$ under the assumption that $U,V \subseteq \pmf$ are simplices in Dehn-Thurston coordinates. This result will be the main counting tool used in the proof of Theorem \ref{theo:main}.

\begin{theorem}
	\cite[Theorem 10.6]{Ara20b}
	\label{theo:bisect_count}
	There exists a constant $\kappa_1 = \kappa_1(g) > 0$ such that for every compact subset $\mathcal{K} \subseteq \mathcal{T}_g$, every $X,Y \in \mathcal{K}$, every set of Dehn-Thurston coordinates $F \colon \mf \to \Sigma^{3g-3}$, every pair of $F$-simplices $U,V \subseteq \pmf$, and every $R > 0$,
	\[
	N(X,Y,U,V,R) = \frac{\nu_X(U) \cdot \nu_Y(V)}{b_g} \cdot e^{(6g-6)R} + O_{\mathcal{K},F}\left(e^{(6g-6-\kappa_1)R}\right).
	\]
\end{theorem}

\subsection*{The tracking principle.} We now discuss the main result of the prequel \cite{Ara20c}, which shows that the action of the mapping class group on the space of closed curves of a closed, orientable surface effectively tracks the corresponding action on Teichmüller space. 

Let $X,Y \in \mathcal{T}_g$, $C > 0$, and $\kappa > 0$. Motivated by Theorem \ref{theo:teich_count}, we say that a subset of mapping classes $\M \subseteq \mcg$ is $(X,Y,C,\kappa)$-sparse if the following bound holds for every $R > 0$,
\[
\#\{\mc \in  \M \ | \ d_\mathcal{T}(X,\mc.Y) \leq R \} \leq C \cdot e^{(6g-6-\kappa)R}.
\] 

Recall that $\mathcal{C}_g$ denotes the space of geodesic currents on $S_g$ and that $i(\cdot,\cdot)$ denotes the geometric intersection number pairing on $\mathcal{C}_g$. The following theorem, the main result of the prequel \cite{Ara20c}, shows that the action of the mapping class group on the space of closed curves of a closed, orientable surface effectively tracks the corresponding action on Teichmüller space in the following sense: for all but quantitatively few mapping classes, the information of how a mapping class moves a given point of Teichmüller space determines, up to a power saving error term, how it changes the geometric intersection numbers of a given closed curve with respect to arbitrary geodesic currents. This result is the cornerstone of the tracking method sketched in \S1 that will be used to prove Theorem \ref{theo:main}.

\begin{theorem}
	\cite[Theorem 9.1]{Ara20c}
	\label{theo:track} 
	There exists a constant $\kappa = \kappa(g) > 0$ such that for every compact subset $\mathcal{K} \subseteq \mathcal{T}_g$ and every closed curve $\beta$ on $S_g$ there exists a constant $C  = C(\mathcal{K},\beta) > 0$ with the following property. For every $X,Y \in \mathcal{K}$ there exists an $(X,Y,C,\kappa)$-sparse subset of mapping classes $M = M(X,Y,\beta) \subseteq \mcg$ such that for every compact subset $K \subseteq \mathcal{C}_g$, every geodesic current $\alpha \in K$, and every $\mc \in \mcg \setminus M$, if $r := d_\mathcal{T}(X,\mc.Y)$, $q_s := q_s(X,\mc.Y)$, and $q_e := q_e(X,\mc.Y)$, then
	\begin{equation}
	\label{eq:X1}
	i(\alpha,\mc.\beta) = i(\alpha,\Re(q_s)) \cdot i(\mc.\beta,\Im(q_e)) \cdot e^r + O_{\mathcal{K},K}\left(e^{(1-\kappa)r}\right).
	\end{equation}
\end{theorem}

Theorem \ref{theo:track} can be interpreted as follows: for all but quantitatively few mapping classes $\mc \in \mcg$, to estimate the geometric intersection number $i(\alpha,\mc.\beta)$, it is not necessary to know how $\alpha$ and $\mc.\beta$ interact between themselves, but rather, it is enough to know how they independently interact with objects determined by the action of $\mc$ on $\mathcal{T}_g$. Let us highlight the fact that the terms $i(\alpha,\Re(q_s))$ and $i(\mc.\beta,\Im(q_e))$ in (\ref{eq:X1}) can be controlled using the information prescribed by the bisector counting function $N(X,Y,U,V,R)$ introduced in (\ref{eq:N_count}).

Recall that $\mathcal{C}_g^* \subseteq \mathcal{C}_g$ denotes the open subset of filling geodesic currents on $S_g$. When proving Theorem \ref{theo:main} we will directly use the following multiplicative version of Theorem \ref{theo:track} for filling geodesic currents and filling closed curves on $S_g$ rather than refering back to the original statement.

\begin{theorem}
	\label{theo:track_mult} 
	There exists a constant $\kappa_2 = \kappa_2(g) > 0$ such that for every compact subset $\mathcal{K} \subseteq \mathcal{T}_g$, every compact subset $K \subseteq \mathcal{C}_g^*$, and every filling closed curve $\beta$ on $S_g$ there exist constants $C  = C(\mathcal{K},\beta) > 0$ and $A = A(\mathcal{K},K,\beta) > 0$ with the following property. For every $X,Y \in \mathcal{K}$ there exists an $(X,Y,C,\kappa_2)$-sparse subset $M = M(X,Y,\beta) \subseteq \mcg$ such that for every $\alpha \in K$ and every $\mc \in \mcg \setminus M$, if $r := d_\mathcal{T}(X,\mc.Y)$, $q_s := q_s(X,\mc.Y)$, and $q_e := q_e(X,\mc.Y)$, then
	\begin{gather*}
	i(\alpha,\mc.\beta) \leq i(\alpha,\Re(q_s)) \cdot i(\mc.\beta,\Im(q_e)) \cdot e^r \cdot (1 + A \cdot e^{-\kappa_2 r}),\\
	i(\alpha,\mc.\beta) \geq i(\alpha,\Re(q_s)) \cdot i(\mc.\beta,\Im(q_e)) \cdot e^r \cdot (1 - A \cdot e^{-\kappa_2 r}).
	\end{gather*}
\end{theorem}

\begin{proof}
	Fix $\mathcal{K} \subseteq \mathcal{T}_g$ compact, $K \subseteq \mathcal{C}_g^*$ compact, and $\beta$ a filling closed curve on $S_g$. Let $X,Y \in \mathcal{K}$ and $\alpha \in K$. Suppose $\mc \in \mcg$ is such that $\mc.Y \neq X$. Denote $q_s:=q_s(X,\mc.Y) \in S(X)$ and $q_e:= q_e(X,\mc.Y) \in S(\mc.Y)$. As $K \subseteq \mathcal{C}_g^*$ is a compact subset of filling geodesic currents on $S_g$, as $\alpha \in K$, as $\mathcal{K} \subseteq \mathcal{T}_g$ is compact, and as $q_s \in S(X) \subseteq \pi^{-1}(\mathcal{K})$,
	\begin{equation}
	\label{eq:e1}
	i(\alpha,\Re(q_s)) \succeq_{\mathcal{K},K} 1.
	\end{equation}
	As $\beta$ is a filling closed curve on $S_g$, as $\mathcal{K} \subseteq \mathcal{T}_g$ is compact, and as $\mc^{-1}.q_e \in S(Y) \subseteq \pi^{-1}(\mathcal{K})$,
	\begin{equation}
	\label{eq:e2}
	i(\mc.\beta,\Im(q_e)) = i(\beta,\Im(\mc^{-1}.q_e)) \succeq_{\mathcal{K},\beta} 1.
	\end{equation}
	Putting together (\ref{eq:e1}) and (\ref{eq:e2}) we deduce
	\begin{equation}
	\label{eq:e3}
	i(\alpha,\Re(q_s)) \cdot i(\mc.\beta,\Im(q_e)) \succeq_{\mathcal{K},K,\beta} 1.
	\end{equation}
	Theorem \ref{theo:track_mult} follows directly from Theorem \ref{theo:track} and (\ref{eq:e3}).
\end{proof}

\subsection*{Sparsity in terms of geometric intersection numbers.} The tracking principle in Theorem \ref{theo:track_mult} holds after discarding a subset of mapping classes that is sparse with respect to the Teichmüller metric. To apply this principle in the proof of Theorem \ref{theo:main} we need to translate this notion of sparsity into a notion defined in terms of the geometric intersection numbers. The following bound, which is a direct consequence of the fact that the Teichmüller geodesic flow contracts the transverse measures of the horizontal foliations of quadratic differentials at an exponential rate, achieves this purpose.

\begin{proposition}
	\label{prop:coarse}
	Let $\mathcal{K} \subseteq \mathcal{T}_g$ be a compact subset and $K \subseteq \mathcal{C}_g^*$ be a compact subset of filling geodesic currents. Then, for every $X,Y \in \mathcal{K}$, every $\alpha,\beta \in K$, and every $\mc \in \mcg$, if $r := d_\mathcal{T}(X,\mc.Y)$, 
	\[
	e^r \preceq_{\mathcal{K},K} i(\alpha,\mc.\beta).
	\]
\end{proposition}

\begin{proof}
	Let $X,Y \in \mathcal{K}$ and $\alpha,\beta \in K$. Suppose that $\mc \in \mcg$ is such that $\mc.\mathcal{K}\cap\mathcal{K} = \emptyset$ and in particular $\mc.Y \neq X$. Let $r := d_\mathcal{T}(X,\mc.Y) > 0$, $q_s:=q_s(X,\mc.Y) \in S(X)$ and $q_e:= q_e(X,\mc.Y) \in S(\mc.Y)$. 
	Recall that $\{a_t\}_{t \in \mathbf{R}}$ denotes the Teichmüller geodesic flow on $\qut$. As $a_r.q_s = q_e$ and as the Teichmüller geodesic flow contracts the transverse measures of the horizontal foliations of quadratic differentials at an exponential rate, $\Im(q_e) = \Im(a_r.q_s) = e^{-r} \cdot \Im(q_s)$. In particular,
	\begin{equation}
	\label{eq:b1}
	i(\mc.\beta,\Im(q_e)) \cdot e^r = i(\mc.\beta,\Im(q_s)).
	\end{equation}
	As $K \subseteq \mathcal{C}_g^*$ is a compact subset of filling geodesic currents on $S_g$, as $\alpha \in K$, as $\mathcal{K} \subseteq \mathcal{T}_g$ is compact, as $q_s \in S(X) \subseteq \pi^{-1}(\mathcal{K})$, and as $\mathcal{C}_g$ is projectively compact, 
	\begin{equation}
	\label{eq:b2}
	i(\mc.\beta,\Im(q_s)) \preceq_{\mathcal{K},K} i(\alpha,\mc.\beta).
	\end{equation}
	As $K \subseteq \mathcal{C}_g^*$ is a compact subset of filling geodesic currents on $S_g$, as $\beta \in K$, as $\mathcal{K} \subseteq \mathcal{T}_g$ is compact, and as $\mc^{-1}.q_e \in S(Y) \subseteq \pi^{-1}(\mathcal{K})$, 
	\begin{equation}
	\label{eq:b3}
	i(\mc.\beta,\Im(q_e)) = i(\beta,\Im(\mc^{-1}.q_e)) \succeq_{\mathcal{K}} 1.
	\end{equation}
	Putting together (\ref{eq:b1}), (\ref{eq:b2}), and (\ref{eq:b3}) we deduce that for every $\mc \in \mcg$ such that $\mc.\mathcal{K} \cap \mathcal{K} = \emptyset$, if $r := d_\mathcal{T}(X,\mc.Y)  > 0$, then
	\begin{equation}
	\label{eq:Y1}
	e^r \preceq_{\mathcal{K},K} i(\alpha,\mc.\beta).
	\end{equation}
	As the action of $\mcg$ on $\mathcal{T}_g$ is properly discontinuous and as $\mathcal{K} \subseteq \mathcal{T}_g$ is compact, there exist only finitely many mapping classes $\mc \in \mcg$ such that $\mc.\mathcal{K} \cap \mathcal{K} \neq \emptyset$. Thus, the estimate in (\ref{eq:Y1})  holds for every $\mc \in \mcg$ by increasing the implicit constant.
\end{proof}

\section{Lipschitz functions on the space of singular measured foliations}

\subsection*{Outline of this section.} In this section we study the regularity of certain functions of interest on the space of singular measured foliations. More specifically, we show that the geometric intersection number with respect to a given geodesic current and the extremal length with respect to a given marked complex structure define Lipschitz functions on the space of singular measured foliations  when parametrized using Dehn-Thurston coordinates. See Theorems \ref{theo:curr_lip} and \ref{theo:ext_lip}. These results will play an important role in the proof of Theorem \ref{theo:main}. These results are proved using train track coordinates and some general facts about convex functions in Euclidean spaces. 

\subsection*{Convexity of geometric intersection numbers.} For the rest of this section we fix an integer $g \geq 2$ and a connected, oriented, closed surface $S_g$ of genus $g$. Recall that $\mf$ denotes the space of singular measured foliations on $S_g$. Let $\tau$ be a maximal train track on $S_g$. Recall that the cone $U(\tau) \subseteq \mf$ of singular measured foliations on $S_g$ carried by $\tau$ can be identified with the cone $V(\tau) \subseteq \smash{(\mathbf{R}_{\geq0})^{18g-18}}$ of non-negative counting measures on the edges of $\tau$ satisfying the switch conditions through a natural $\mathbf{R}^+$-equivariant bijection $\Phi_\tau \colon U(\tau) \to V(\tau)$ we refer to as the train track chart induced by $\tau$ on $\mf$. Recall that $\mathcal{C}_g$ denotes the space of geodesic currents on $S_g$ and that $i(\cdot,\cdot)$ denotes the geometric intersection number pairing on $\mathcal{C}_g$. The following theorem of Mirzakhani shows that geometric intersection numbers are convex in train track coordinates.

\begin{theorem}
	\cite[Theorem A.1]{Mir04}
	\label{theo:current_convex}
	Let $\tau$ be a maximal train track on $S_g$, $\Phi_\tau \colon U(\tau) \to V(\tau)$ be the  train track chart induced by $\tau$ on $\mf$, and $\alpha \in \mathcal{C}_g$ be a geodesic current on $S_g$. Then, the composition $i(\alpha,\cdot) \circ \Phi_\tau^{-1} \colon V(\tau) \to \mathbf{R}$ is convex.
\end{theorem}

\subsection*{Convexity of extremal lengths.} Recall that $\mathcal{T}_g$ denotes the Teichmüller space of marked complex structures on $S_g$. Recall that for every $X \in \mathcal{T}_g$ we denote by $\SExt_X \colon \mf \to \mathbf{R}$ the function which to every $\lambda \in \mf$ assigns the square root of the extremal length of $\lambda$ with respect to $X$. Our next goal is to show these functions are convex in train track coordinates.
These functions cannot be represented using geodesic currents. In particular, Theorem \ref{theo:current_convex} does not apply to them. To prove these functions are convex we draw inspiration from arguments in \cite{Mir04} and \cite{MT20}.

In the following discussion we distinguish between parametrized closed multi-curves and their homotopy classes. In particular, adopting the convention in \cite{MT20}, we refer to a finite collection $C := (c_i)_{i=1}^k$ of parametrized closed curves on $S_g$ as a concrete multi-curve. Two concrete multi-curves on $S_g$ are said to be homotopic if there exists a bijection between their homotopically nontrivial components that pairs components equivalent up homotopy and reparametrization. 

Let $C$ be a concrete multi-curve on $S_g$. We refer to any intersection between the components of $C$ as a crossing of $C$. Disregarding orientations, any intersection between the components of $C$ can be smoothed out in two different ways. See Figure \ref{fig:smoothing}. We refer to any concrete multi-curve obtained by smoothing out a sequence of crossings of $C$ as a smoothing of $C$. 

\begin{figure}[h]
	\centering
	\vspace{+0.23cm}
	\begin{subfigure}[b]{0.3\textwidth}
		\centering
		\includegraphics[width=0.5\textwidth]{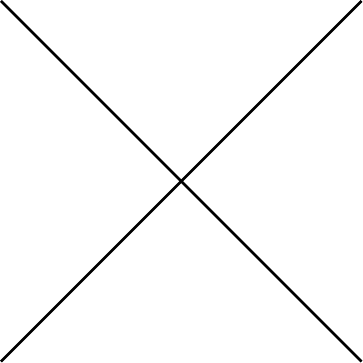}
		\caption{Crossing.}
	\end{subfigure}
	\quad \quad 
	~ 
	\begin{subfigure}[b]{0.3\textwidth}
		\centering
		\includegraphics[width=0.5\textwidth]{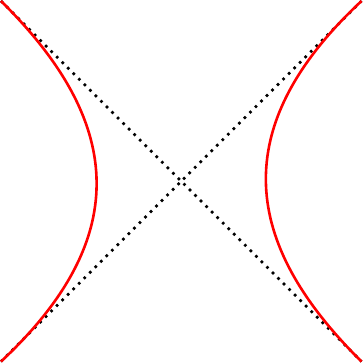}
		\caption{Smoothing I.}
	\end{subfigure}
	\quad \quad 
	~ 
	\begin{subfigure}[b]{0.3\textwidth}
		\centering
		\includegraphics[width=0.5\textwidth]{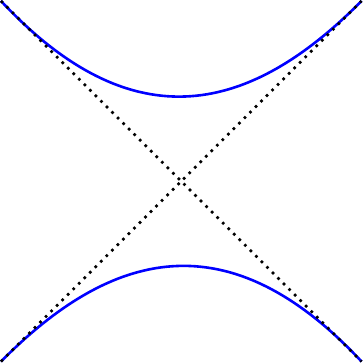}
		\caption{Smoothing II.}
	\end{subfigure}
	\caption{The two different smoothings of a crossing of a concrete multi-curve.} 
	\label{fig:smoothing}
\end{figure}

A concrete multi-curve is said to be essential if it has no null-homotopic components. A concrete multi-curve is said to be self-transverse if all of its components are transverse to themselves and to every other component. Concrete multi-curves can be transformed by performing ambient isotopies, i.e., isotopies of the surface, and Reidemeister moves in embedded disks of the surface. See Figure \ref{fig:red} for pictures of the different Reidemeister moves. A Reidemeister move is said to be performed in a simplifying direction if it does not increase the number of intersections of the arcs involved.

\begin{figure}[h!]
	\centering
	\begin{subfigure}[b]{1.0\textwidth}
		\centering
		\begin{tabular}{c c c}
			\includegraphics[scale=.8]{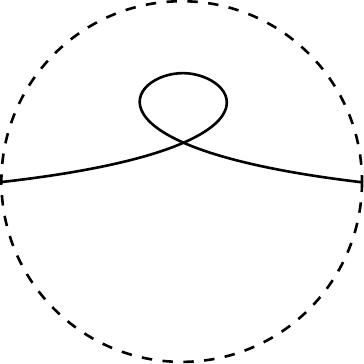} & \begin{tabular}{c} $\rightarrow$  \\[2.69cm] \end{tabular}& \begin{tabular}{c} \includegraphics[scale=.8]{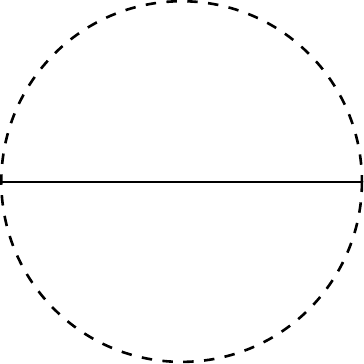} \\[2.6cm] \end{tabular}
		\end{tabular}
	\vspace{-2.5cm}
	\caption{Reidemeister I move.} \vspace{+0.4cm}
	\end{subfigure}
	\begin{subfigure}[b]{1.0\textwidth}
		\centering
		\begin{tabular}{c c c}
			\includegraphics[scale=.8]{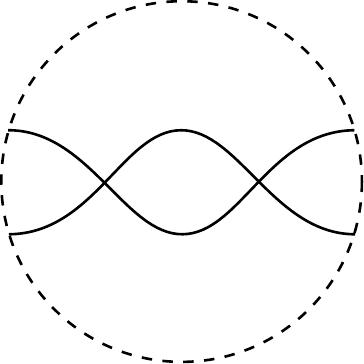} & \begin{tabular}{c} $\rightarrow$  \\[2.69cm] \end{tabular}& \begin{tabular}{c} \includegraphics[scale=.8]{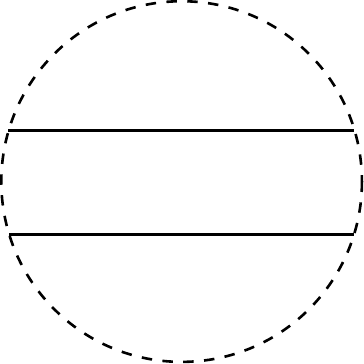} \\[2.6cm] \end{tabular}
		\end{tabular}
		\vspace{-2.5cm}
		\caption{Reidemeister II move.} \vspace{+0.4cm}
	\end{subfigure}
	\begin{subfigure}[b]{1.0\textwidth}
		\centering
		\begin{tabular}{c c c}
			\includegraphics[scale=.8]{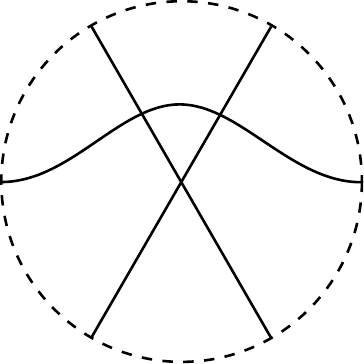} & \begin{tabular}{c} $\leftrightarrow$  \\[2.69cm] \end{tabular}& \begin{tabular}{c} \includegraphics[scale=.8]{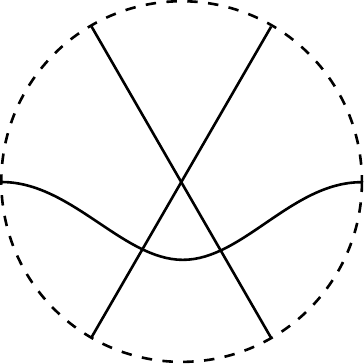} \\[2.6cm] \end{tabular}
		\end{tabular}
		\vspace{-2.5cm}
		\caption{Reidemeister III move.}
	\end{subfigure}
	\captionsetup{width=\linewidth}
	\caption{Reidemeister moves with arrows representing the simplifying directions.} \label{fig:red} 
\end{figure}

\begin{theorem}
	\cite[Theorem 2.2]{HS94} \cite[Theorem 1]{dG97}
	\label{theo:red}
	Let $C$ be an essential and self-transverse concrete multi-curve on $S_g$ minimizing the number of self-crossings in its homotopy class and having no pair of components homotopic to distinct powers of the same closed curve. Then, for every essential and self-transverse concrete multi-curve $C'$ on $S_g$ homotopic to $C$ there exists a finite sequence of ambient isotopies and Reidemeister I, II, and III moves that transform $C'$ into $C$, with the Reidemeister I and II moves being performed only in the simplifying direction.
\end{theorem}

Using Theorem \ref{theo:red} we deduce the following result which shows that the smoothings of a minimally self-intersecting concrete multi-curve are witnessed by all homotopic concrete multi-curves.

\begin{lemma}
	\label{lem:smooth}
	Let $C$ be a self-transverse concrete multi-curve on $S_g$ minimizing the number of self-crossings in its homotopy class and having no pair of components homotopic to distinct powers of the same closed curve. Suppose $C'$ is a smoothing of $C$. Then, every concrete multi-curve $D$ on $S_g$ homotopic to $C$ has a smoothing $D'$ homotopic to $C'$.
\end{lemma}

\begin{proof}
	Without loss of generality we can assume that $C$ and $D$ are essential and that $D$ is self-transverse. Indeed, the lemma does not change if null-homotopic components are discarded and self-transversality can be achieved through an arbitrarily small perturbation that does not introduce new crossings. By Theorem \ref{theo:red} there exists a finite sequence of ambient isotopies and Reidemeister I, II, and III moves that transform $D$ into $C$, with the Reidemeister I and II moves being performed only in the simplifying direction. The lemma can then be proved by following this sequence of transformations backwards from $C$ to $D$ and showing that after each step one can find a sequence of smoothings that yields a concrete multi-curve homotopic to $C'$. The only case which requires a non-trivial analysis is that of Reidemeister III moves. Checking the $27$ different possible smoothings of the arcs involved in a Reidemeister III move finishes the proof. See Figure \ref{fig:smooth} for an example.
\end{proof}

\begin{figure}[h]
	\centering
	\begin{subfigure}[b]{0.4\textwidth}
		\centering
		\includegraphics[width=0.5\textwidth]{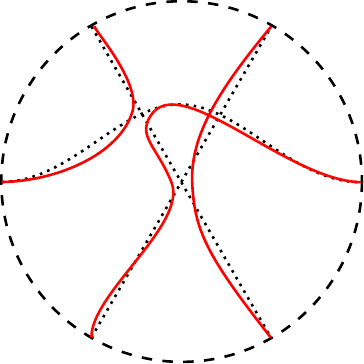}
		\caption{Before Reidemeister III move.}
	\end{subfigure}
	\quad \quad \quad
	~ 
	\begin{subfigure}[b]{0.4\textwidth}
		\centering
		\includegraphics[width=0.5\textwidth]{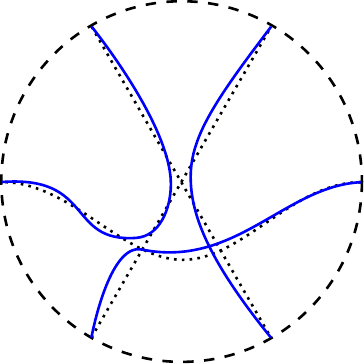}
		\caption{After Reidemeister III move.}
	\end{subfigure}
	\caption{Homotopic smoothings of the arcs involved in a Reidemeister III move.} 
	\label{fig:smooth}
\end{figure}

Let $\alpha$ and $\beta$ be a pair of parametrized simple closed curves on $S_g$ intersecting transversely. We say $\alpha$ and $\beta$ intersect minimally if they realize the minimum number of intersections among pairs of closed curves homotopic to $\alpha$ and $\beta$. We say $\alpha$ and $\beta$ form a bigon if there exists an embedded closed disk in $S_g$ whose boundary is the union of an arc of $\alpha$ and an arc of $\beta$ intersecting at exactly two points. The following well known result provides a criterion for checking when a pair of parametrized simple closed curves intersect minimally.
 
\begin{proposition}
	\label{prop:bigon_crit}
	\cite[Proposition 1.7]{FM11}
	Let $\alpha$ and $\beta$ be a pair of parametrized simple closed curves on $S_g$ intersecting transversely and forming no bigons. Then $\alpha$ and $\beta$ intersect minimally.
\end{proposition}

We now extend the definition of $\SExt_X$ to concrete multi-curves. Let $X \in \mathcal{T}_g$ be a marked complex structure on $S_g$. Fix a Riemannian metric $h$ on $S_g$ in the conformal class of $X$. Denote by $dA_h$ and $ds_h$ the area and length elements of $h$. Every metric on $S_g$ in the conformal class of $X$ can be obtained as a product $\rho h$ for some non-negative measurable function $\rho \colon S_g \to \mathbf{R}_{\geq 0}$. For any such function define
\[
\mathrm{Area}(\rho h) := \int_X \rho^2 \thinspace dA_h.
\]
Let $C$ be a concrete multi-curve on $S_g$. For every non-negative measurable function $\rho \colon S_g \to \mathbf{R}$ define
\[
\ell_{\rho h}(C) := \inf_{D \sim C} \thinspace \sum_{i=1}^k \thinspace \int_{d_i} \rho \thinspace ds_h,
\]
where the infimum runs over all concrete multi-curves $D := (d_i)_{i=1}^k$ homotopic to $C$. The square root of the extremal length of $C$ with respect to $X$ is defined as
\[
\SExt_X(C) := \sup_{\rho} \frac{\ell_{\rho h}(C)}{\sqrt{\mathrm{Area}(\rho h)}},
\]
where the supremum runs over all non-negative measurable functions $\rho \colon S_g \to \mathbf{R}_{\geq 0}$ such that $0 < \mathrm{Area}(\rho h)<+\infty$. The following result, which is a direct consequence of the definitions and Lemma \ref{lem:smooth}, describes the behavior of $\SExt_X$ under the operation of smoothing out some of the crossings of a minimally self-intersecting concrete multi-curve. For a related result see \cite[Lemma 4.16]{MT20}.

\begin{lemma}
	\label{lem:ext_smooth}
	Let $X \in \mathcal{T}_g$ and $C$ be a self-transverse concrete multi-curve on $S_g$ minimizing the number of self-crossings in its homotopy class and having no pair of components homotopic to distinct powers of the same closed curve. Then, for every smoothing $C'$ of $C$,
	\[
	\SExt_X(C') \leq \SExt_X(C).
	\]
\end{lemma}

\begin{proof}
	Let $h$ be a metric in the conformal class of $X \in \mathcal{T}_g$ and $C'$ be a smoothing of the concrete multi-curve $C$ on $S_g$. Fix a non-negative measurable function $\rho \colon S_g \to \mathbf{R}_{\geq 0}$ such that $\mathrm{Area}(\rho h) > 0$. Let $\epsilon > 0$ be arbitrary. Consider a concrete multi-curve $D := (d_i)_{i=1}^k$  homotopic to $C$ such that
	\[
	\sum_{i=1}^k \thinspace \int_{d_i} \rho \thinspace ds_h \leq \ell_{\rho h}(C) + \epsilon.
	\]
	By Lemma \ref{lem:smooth}, $D$ has a smoothing $D'$ homotopic to $C'$. As smoothing out some of the crossings of a concrete multi-curve on $S_g$ does not increase its length with respect to the metric $\rho h$ we deduce
	\[
	\ell_{\rho h}(C') \leq \sum_{i=1}^k \thinspace \int_{d_i} \rho \thinspace ds_h \leq \ell_{\rho h}(C) + \epsilon.
	\]
	As $\epsilon > 0$ is arbitrary it follows that
	\[
	\ell_{\rho h}(C') \leq \ell_{\rho h}(C).
	\]
	Dividing by $\sqrt{\mathrm{Area}(\rho h)} > 0$ we deduce
	\[
	\frac{\ell_{\rho h}(C')}{\sqrt{\mathrm{Area}(\rho h)}} \leq \frac{\ell_{\rho h}(C)}{\sqrt{\mathrm{Area}(\rho h)}}.
	\]
	Taking supremum over all measurable functions $\rho \colon S_g \to \mathbf{R}_{\geq 0}$ such that $\mathrm{Area}(\rho h) > 0$ we conclude
	\[
	\SExt_X(C') \leq \SExt_X(C). \qedhere
	\]
\end{proof}

Given a pair of concrete multi-curves $C_1$ and $C_2$ on $S_g$ denote by $C_1 \cup C_2$ the concrete multi-curve obtained by taking the union of the components of $C_1$ and $C_2$. The following result, which is a direct consequence of the definitions, describes the effect on $\SExt_X$ of applying this operation. 

\begin{lemma}
	\label{lem:ext_cup}
	\cite[Lemma 4.17]{MT20}
	Let $X \in \mathcal{T}_g$ and $C_1,C_2$ be concrete multi-curves on $S_g$. Then,
	\[
	\SExt_X(C_1 \cup C_2) \leq \SExt_X(C_1) + \SExt_X(C_2).
	\]
\end{lemma}

\begin{proof}
	Let $h$ be a metric in the conformal class of $X \in \mathcal{T}_g$.  Fix a non-negative measurable function $\rho \colon S_g \to \mathbf{R}_{\geq 0}$ such that $\mathrm{Area}(\rho h) > 0$. Directly from the definitions it follows that
	\[
	\ell_{\rho h}(C_1 \cup C_2) = \ell_{\rho h}(C_1) + \ell_{\rho h}(C_2).
	\]
	Dividing by $\sqrt{\mathrm{Area}(\rho h)} > 0$ we deduce
	\[
	\frac{\ell_{\rho h}(C_1 \cup C_2)}{\sqrt{\mathrm{Area}(\rho h)}} \leq \frac{\ell_{\rho h}(C_1)}{\sqrt{\mathrm{Area}(\rho h)}} + \frac{\ell_{\rho h}(C_2)}{\sqrt{\mathrm{Area}(\rho h)}}.
	\]
	Taking supremum over all measurable functions $\rho \colon S_g \to \mathbf{R}_{\geq 0}$ such that $\mathrm{Area}(\rho h) > 0$ we conclude
	\[
	\SExt_X(C_1 \cup C_2) \leq \SExt_X(C_1) + \SExt_X(C_2). \qedhere
	\]
\end{proof}

We are now ready to prove that the functions $\SExt_X \colon \mf \to \mathbf{R}$ are convex in train track coordinates. It will be important to recall that these functions are homogeneous with respect to the natural $\mathbf{R}^+$ scaling action on $\mf$. The following proof is inspired by arguments introduced by Mirzakhani in the proof of \cite[Theorem A.1]{Mir04}.

\begin{theorem}
	\label{theo:ext_convex}
	Let $\tau$ be a maximal train track on $S_g$, $\Phi_\tau \colon U(\tau) \to V(\tau)$ be the  train track chart induced by $\tau$ on $\mf$, and $X \in \mathcal{T}_g$. Then, the composition $\SExt_X \circ \Phi_\tau^{-1} \colon V(\tau) \to \mathbf{R}$ is convex.
\end{theorem}

\begin{proof}
	As the composition $\SExt_X \circ \Phi_\tau^{-1} \colon V(\tau) \to \mathbf{R}$ is homogeneous, to show it is convex, it is enough to verify that for every pair of counting measures $u,v \in V(\tau)$,
	\begin{equation}
	\label{eq:Z1}
	\SExt_X \circ \Phi_\tau^{-1}(u+v) \leq \SExt_X \circ \Phi_\tau^{-1} (u) + \SExt_X \circ \Phi_\tau^{-1} (v).
	\end{equation}
	As the set of rationally weighted simple closed curves carried by $\tau$ is dense in $U(\tau)$, it is enough to verify (\ref{eq:Z1}) for pairs of counting measures $u,v \in V(\tau)$ coming from this set. Furthermore, as $\SExt_X \circ \Phi_\tau^{-1} \colon V(\tau) \to \mathbf{R}$ is homogeneous, it is enough to verify (\ref{eq:Z1}) for pairs of counting measures $u,v \in V(\tau)$ coming from integral multiples of simple closed curves carried by $\tau$. 
	
	Let $\alpha,\beta \in U(\tau)$ be a pair of simple closed curves carried by $\tau$ and $n,m \in \mathbf{N}^+$. Consider the counting measures $u  := \Phi_\tau(n \alpha) \in V(\tau)$ and $v:= \Phi_\tau(m \beta) \in V(\tau)$. To show (\ref{eq:Z1}) holds we will explicitly construct a concrete multi-curve in the homotopy class of $\Phi_\tau^{-1}(u+v) \in U(\tau)$. Consider a neighborhood $W \subseteq S_g$ of $\tau$ foliated by transverse ties as in Figure \ref{fig:ties}. Extend this foliation to a singular foliation $\mathcal{F}$ of $S_g$ with one three-pronged singularity in each complementary trigon of $\tau$. Realize $\alpha$ and $\beta$ as parametrized simple closed curves in $W$ transverse to each other and to the foliation by ties.
	
	\begin{figure}[h!]
		\centering
		\includegraphics[width=.25\textwidth]{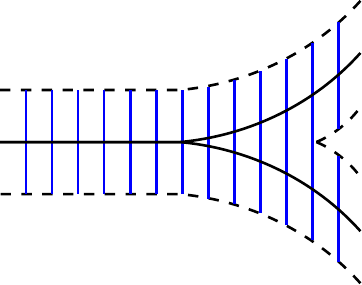}
		\caption{Neighborhood of a train track foliated by ties.} \label{fig:ties} 
	\end{figure}
	
	We claim that any bigons formed by $\alpha$ and $\beta$ must lie completely inside $W$. Indeed, if this was not the case then $\alpha$ and $\beta$ would form a bigon containing a singularity of $\mathcal{F}$ and through the procedure described in Figure \ref{fig:bigon} one would be able to construct a singular foliation of a torus with at least one three-pronged singularity, contradicting the Poincaré-Hopf theorem. 
	
	\begin{figure}[h]
		\centering
		\begin{subfigure}[b]{0.4\textwidth}
			\centering
			\includegraphics[width=0.6\textwidth]{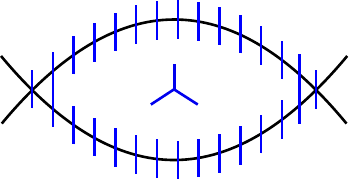} \vspace{+0.3cm}
			\caption{Foliated bigon.}
			\label{fig:A}
		\end{subfigure} 
		\quad \quad \quad
		~ 
		\begin{subfigure}[b]{0.4\textwidth}
			\centering
			\includegraphics[width=0.4\textwidth]{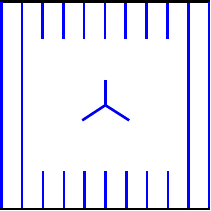}
			\caption{Foliated square.}
			\label{fig:B}
		\end{subfigure}
		\caption{Constructing a singular foliation of a torus with at least one three-pronged singularity. Deform a foliated bigon as the one in Figure \ref{fig:A} into a foliated square as the one in Figure \ref{fig:B} and identify opposite sides of the square to get a foliated torus.} 
		\label{fig:bigon}
	\end{figure}
	
	As all bigons formed by $\alpha$ and $\beta$ lie completely inside $W$, one can isotope out all such bigons, starting with the innermost one, to realize $\alpha$ and $\beta$ as parametrized simple closed curves in $W$ transverse to each other, to the foliation by ties, and forming no bigons. In particular, by Proposition \ref{prop:bigon_crit}, such realizations of $\alpha$ and $\beta$ intersect minimally. Consider $n$ disjoint parallel copies of $\alpha$ and $m$ disjoint parallel copies of $\beta$ obtained by translating $\alpha$ and $\beta$ a small distance in the direction of the ties so that any pair of copies of $\alpha$ and $\beta$ are transverse to each other and intersect minimally. Let $C =  n\alpha \cup m\beta$ be the concrete multi-curve obtained by taking the $\cup$ operation over the $n$ parallel copies of $\alpha$ and $m$ parallel copies of $\beta$. Denote by $C'$ the concrete multi-curve obtained by smoothing out all the crossings of $C$ in the direction transverse to the ties as in Figure \ref{fig:tie_smooth}. By Lemmas \ref{lem:ext_smooth} and \ref{lem:ext_cup},
	\begin{equation}
	\label{eq:Z}
	\SExt_X(C') \leq \SExt_X(C) \leq \SExt_X(n\alpha) + \SExt_X(m\beta).
	\end{equation}
	
	\begin{figure}[h]
		\centering
		\begin{subfigure}[b]{0.4\textwidth}
			\centering
			\includegraphics[width=0.45\textwidth]{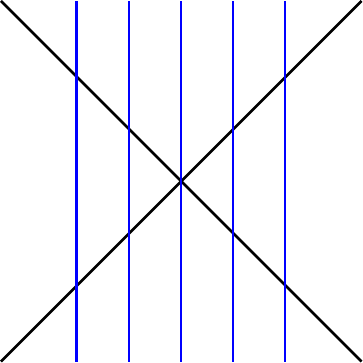}
			\caption{Before smoothing out.}
		\end{subfigure}
		\quad \quad \quad
		~ 
		\begin{subfigure}[b]{0.4\textwidth}
			\centering
			\includegraphics[width=0.45\textwidth]{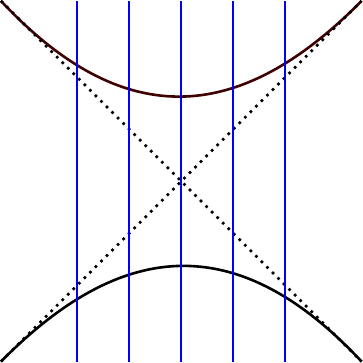}
			\caption{After smoothing out.}
		\end{subfigure}
		\captionsetup{width=\linewidth}
		\caption{Smoothing out a crossing in the direction transverse to the ties.} 
		\label{fig:tie_smooth}
	\end{figure}
	
	The concrete multi-curve $C'$ is carried by $\tau$ and its corresponding counting measure is precisely $u+v \in V(\tau)$. From (\ref{eq:Z}) we conclude (\ref{eq:Z1}) holds, thus finishing the proof of the theorem.
\end{proof}

\subsection*{Convex functions are Lipschitz.} We now discuss some general facts about convex functions in Euclidean spaces. Let $\Omega \subseteq \mathbf{R}^n$ be an open convex subset and $K \subseteq \Omega$ be a compact subset. Denote by $d(K,\partial\Omega)$ the Euclidean distance between $K$ and $\partial \Omega \subseteq \mathbf{R}^n$ with the convention that $d(K,\partial\Omega) = +\infty$ if $\Omega = \mathbf{R}^n$. Let $r = r(\Omega,K) := \min\{1,d(K,\partial\Omega) /2\}$. Denote by $\mathrm{Nbhd}_r(K) \subseteq \Omega$ the set of points in $\mathbf{R}^n$ at Euclidean distance at most $r$ from $K$. Let $f \colon \Omega \to \mathbf{R}$ be a convex function. Denote 
\[
L = L(\Omega,K,f) := \textstyle \smash{2 \cdot r^{-1} \cdot \sup_{x \in \mathrm{Nbhd}_r(K)} |f(x)|}.
\]

The following well known result shows that convex functions on Euclidean spaces are Lipschitz when restricted to compact subsets of their domain.

\begin{proposition}
	\label{prop:convex_lip}
	Let $\Omega \subseteq \mathbf{R}^n$ be an open, convex subset and $f \colon \Omega \to \mathbf{R}$ be a convex function. Then, for every $K \subseteq \Omega$ compact, the restriction $f|_K \colon K \to \mathbf{R}$ is $L$-Lipschitz for $L = L(\Omega,K,f)$ as above.
\end{proposition}

Let $C \subseteq \smash{(\mathbf{R}_{>0})^n}$ be a cone cut out by homogeneous linear equations and $D \subseteq C$ be a projectively compact cone. Denote by $D' \subseteq  \smash{(\mathbf{R}_{>0})^n}$ the set of points in $D$ with Euclidean norm between $1$ and $2$. Let $\partial C \subseteq \smash{(\mathbf{R}_{\geq 0})^n}$ be the boundary of $C$ when considered as a subset of $\smash{(\mathbf{R}_{\geq 0})^n}$. Denote by $d(D',\partial C)$ the Euclidean distance between $D'$ and $\partial C$. Let $r' = r'(C,D) := \min\{1,d(D',\partial C) /2\}$. Denote by $\mathrm{Nbhd}_{r'}(D') \subseteq \smash{(\mathbf{R}_{>0})^n}$ the set of points in $ \smash{(\mathbf{R}_{>0})^n}$ at Euclidean distance at most $r'$ from $D'$. Let $f \colon C \to \mathbf{R}$ be a convex, homogeneous function. Denote 
\[
L' = L'(C,D,f) := \textstyle \smash{2 \cdot (r')^{-1} \cdot \sup_{x \in \mathrm{Nbhd}_{r'}(D')} |f(x)|}.
\]

Directly from Proposition \ref{prop:convex_lip} we deduce the following result.

\begin{proposition}
	\label{prop:convex_lip_2}
	Let $C \subseteq \smash{(\mathbf{R}_{>0})^n}$ be a cone cut out by homogeneous linear equations and $f \colon C \to \mathbf{R}$ be a convex, homogeneous function. Then, for every projectively compact cone $D \subseteq C$, the restriction $f|_{D} \colon D \to \mathbf{R}$ is $L'$-Lipschitz for $L' = L'(C,D,f)$ as above.
\end{proposition}

\subsection*{Lipschitz functions in Dehn-Thurston coordinates.} We are now ready to prove the main results of this section. Recall that $\Sigma$ denotes the piecewise linear manifold $\IT:= \mathbf{R}^2 / \langle-1\rangle$ endowed with the quotient Euclidean metric. Recall that the product $\IT^{3g-3}$ is a piecewise linear manifold which we endow with the $L^2$ product metric. Recall that any set of Dehn-Thurston coordinates provides a homeomorphism $F \colon \mf \to \smash{\IT^{3g-3}}$ equivariant with respect to the natural $\mathbf{R}^+$ scaling actions on $\mf$ and $\smash{\IT^{3g-3}}$. Recall that the change of coordinates maps between Dehn-Thurston coordinates and train track coordinates are piecewise linear. The following result shows that geometric intersection numbers with respect to arbitrary geodesic currents are Lipschitz in Dehn-Thurston coordinates.

\begin{theorem}
	\label{theo:curr_lip}
	Let $F \colon \mf \to \smash{\IT^{3g-3}}$ be a set of Dehn-Thurston coordinates of $\mf$ and $K \subseteq \mathcal{C}_g$ be a compact subset of geodesic currents on $S_g$. Then, there exists a constant $L  = L(F,K) > 0$ such that for every $\alpha \in K$ the composition $i(\alpha,\cdot) \circ F^{-1} \colon \smash{\IT^{3g-3}} \to \mathbf{R}$ is $L$-Lipschitz.
\end{theorem}

\begin{proof}
	Given a set of Dehn-Thurston coordinates $F \colon \mf \to \IT^{3g-3}$ of $\mf$ denote by $d_F$ the metric on $\mf$ induced through this identification and by $\ell_F(\gamma)$ the length of a piecewise smooth path $\gamma \colon [0,1] \to \mf$ with respect to this metric. Given a maximal train track $\tau$ on $S_g$ denote by $d_{\tau}$ the Euclidean metric on $U(\tau) \subseteq \mf$ induced through the train track chart $\Phi_\tau \colon U(\tau) \to V(\tau)$ and by $\ell_\tau(\gamma)$ the length of a piecewise smooth path $\gamma \colon [0,1] \to U(\tau)$ with respect to this metric. 
	
	Every singular measured foliation $\lambda \in \mf$ is carried by a maximal train track $\tau$ on $S_g$ in such a way that $\lambda$ belongs to the interior of $U(\tau) \subseteq \mf$. In particular, as $\mf$ is projectively compact, there exist a finite collection $\{\tau_i\}_{i=1}^n$ of maximal train tracks on $S_g$ and a finite collection $\{W_i\}_{i=1}^n$ of open subsets of $\mf$ such that $W_i$ is compactly contained in the interior of $U(\tau_i) \subseteq \mf$ for every $i \in \{1,\dots,n\}$ and such that $\mf \subseteq \bigcup_{i=1}^n W_i$. 
	
	Fix  a set of Dehn-Thurston coordinates $F \colon \mf \to \IT^{3g-3}$ of $\mf$, a compact subset $K \subseteq \mathcal{C}_g$, and a geodesic current $\alpha \in K$. Let $\lambda_0,\lambda_1 \in \mf$ be arbitrary. Our goal is to show that
	\[
	|i(\alpha,\lambda_0) - i(\alpha,\lambda_1)| \preceq_{F,K} d_F(\lambda_0,\lambda_1).
	\]
	
	Consider a piecewise smooth path $\gamma \colon [0,1] \to \mf$ such that $\gamma(0) = \lambda_1$, $\gamma(1) = \lambda_1$, and $\ell_F(\gamma) \leq 2 d_F(\lambda_0,\lambda_1)$. By the Lebesgue number lemma, there exists a finite partition $0 =: t_0 < t_1 < \cdots < t_{m-1} < t_m := 1$ such that for every $j \in \{0,\dots,m-1\}$ there exists $i(j) \in \{1,\dots,n\}$ satisfying $\gamma([t_{j},t_{j+1}]) \subseteq W_{i(j)}$. As the change of coordinates maps between Dehn-Thurston coordinates and train track coordinates are piecewise linear, the following estimate holds for every $j \in \{0,\dots,m-1\}$,
	\begin{equation}
	\label{eq:D1}
	\ell_{\tau_{i(j)}}\left(\gamma|_{[t_{j},t_{j+1}]}\right) \asymp_{F} \ell_{F}\left(\gamma|_{[t_{j},t_{j+1}]}\right).
	\end{equation}
	As a consequence of Theorem \ref{theo:current_convex} and Proposition \ref{prop:convex_lip_2} we deduce that for every $j \in \{0,\dots,m-1\}$,
	\begin{equation}
	\label{eq:D2}
	|i(\alpha,\gamma(t_j)) - i(\alpha,\gamma(t_{j+1}))| \preceq_{K} d_{\tau_{i(j)}}(\gamma(t_j), \gamma(t_{j+1})) \leq \ell_{\tau_{i(j)}}\left(\gamma|_{[t_{j},t_{j+1}]}\right).
	\end{equation}
	Using the triangle inequality, (\ref{eq:D2}), (\ref{eq:D1}), and the assumption that $\ell_F(\gamma) \leq 2 d_F(\lambda_0,\lambda_1)$, we conclude
	\[
	|i(\alpha,\lambda_0) - i(\alpha,\lambda_1)| \leq \sum_{j=0}^{m-1} |i(\alpha,\gamma(t_j)) - i(\alpha,\gamma(t_{j+1}))| \preceq_{F,K} d_F(\lambda_0,\lambda_1). \qedhere
	\]
\end{proof}

The arguments introduced in the proof of Theorem \ref{theo:curr_lip} can also be used to study extremal lengths by applying Theorem \ref{theo:ext_convex} in place of Theorem \ref{theo:current_convex}. More concretely, one can show that extremal lengths are Lipschitz in Dehn-Thurston coordinates in the following sense.

\begin{theorem}
	\label{theo:ext_lip}
	Let $F \colon \mf \to \smash{\IT^{3g-3}}$ be a set of Dehn-Thurston coordinates of $\mf$ and $\mathcal{K} \subseteq \mathcal{T}_g$ be a compact subset of marked complex structures on $S_g$. Then, there exists a constant $L = L(F,\mathcal{K}) > 0$ such that for every $X \in \mathcal{K}$ the composition $\SExt_X \circ F^{-1} \colon \smash{\IT^{3g-3}} \to \mathbf{R}$ is $L$-Lipschitz.
\end{theorem}

\subsection*{Lipschitz functions on the space of projective singular measured foliations.} We finish this section by discussing an important consequence of Theorem \ref{theo:ext_lip}. When proving Theorem \ref{theo:main} we will directly apply this consequence rather than refer back to the original result.

Recall that $\pmf$ denotes the space of projective singular measured foliations on $S_g$. Recall that the projectivization $P\Sigma^{3g-3}$ of $\smash{\Sigma^{3g-3}}$ can be identified with subset $\smash{\Sigma_u^{3g-3}} \subseteq \smash{\Sigma^{3g-3}}$ of points in $\smash{\Sigma^{3g-3}}$ of unit $L^1$-norm through a homeomorphism we denote by $u \colon P\Sigma^{3g-3} \to \smash{\Sigma_u^{3g-3}}$. Recall that we endow the subset $\smash{\Sigma_u^{3g-3}} \subseteq \Sigma^{3g-3}$ with the induced Euclidean metric.  Recall that given a set of Dehn-Thurston coodinates $F \colon \mf \to \smash{\Sigma^{3g-3}}$ we denote by $PF \colon \pmf \to P\Sigma^{3g-3}$ its projectivization.

Let $X \in \mathcal{T}_g$ be a marked complex structure on $S_g$. Recall that the function $\SExt_X \colon \mf \to \mathbf{R}$ is homogeneous with respect to the natural action of $\mathbf{R}^+$ on $\mf$ which scales transverse measures. Consider the map $E_X \colon \pmf \to \mf$ which to every projective class $[\lambda] \in \pmf$ assigns the unique singular measured foliation $E_X([\lambda]) := \lambda/\SExt_X(\lambda) \in \mf$ in the projective class $[\lambda]$ of unit extremal length with respect to $X$. The following result, which is a direct consequence of Theorem \ref{theo:ext_lip}, shows that the map $E_X \colon \pmf \to \mf$ is Lipschitz in Dehn-Thurston coordinates.

\begin{theorem}
	\label{theo:ext_lip_2}
	Let $F \colon \mf \to \smash{\IT^{3g-3}}$ be a set of Dehn-Thurston coordinates of $\mf$ and $\mathcal{K} \subseteq \mathcal{T}_g$ be a compact subset of marked complex structures on $S_g$. Then, there exists a constant $L = L(F,\mathcal{K}) > 0$ such that for every $X \in \mathcal{K}$ the composition $F \circ E_X  \circ PF^{-1}\circ u^{-1} \colon \smash{\Sigma^{3g-3}_u} \to \smash{\Sigma^{3g-3}}$ is $L$-Lipschitz.
\end{theorem}

\begin{proof}
	Denote by $d_F$ be the metric on $\mf$ induced through the identification $F \colon \mf \to \IT^{3g-3}$.  This metric is homogeneous with respect to the natural scaling action of $\mathbf{R}^+$ on $\mf$. Denote by $\|\cdot \|_F$ the pullback of the Euclidean norm on $\IT^{3g-3}$ through the identification $F \colon \mf \to \IT^{3g-3}$. Notice $\| \lambda \|_F \preceq 1$ for every $\lambda \in F^{-1}(\Sigma^{3g-3}_u) \subseteq \mf$.
	The compactness of $\mathcal{K} \subseteq \mathcal{T}_g$ and $F^{-1}(\Sigma^{3g-3}_u) \subseteq \mf$ ensures the following estimate holds for every $X \in \mathcal{K}$ and every $\lambda \in F^{-1}(\Sigma^{3g-3}_u)$,
	\begin{equation}
	\label{eq:C0}
	\SExt_X(\lambda) \asymp_{F,\mathcal{K}} 1.
	\end{equation}
	
	Let $X \in \mathcal{K}$ and $\lambda_0,\lambda_1 \in F^{-1}(\Sigma^{3g-3}_u)$ be arbitrary. Our goal is to show that
	\[
	d_F\left(\frac{\lambda_0}{\SExtn_X(\lambda_0)}, \frac{\lambda_1}{\SExtn_X(\lambda_1)}\right) \preceq_{F,K} d_F(\lambda_0,\lambda_1).
	\]
	By the triangle inequality,
	\begin{gather}
	\label{eq:C1}
	d_F\left(\frac{\lambda_0}{\SExtn_X(\lambda_0)}, \frac{\lambda_1}{\SExtn_X(\lambda_1)}\right) \\
	\leq  d_F\left(\frac{\lambda_0}{\SExtn_X(\lambda_0)}, \frac{\lambda_1}{\SExtn_X(\lambda_0)}\right) +  d_F\left(\frac{\lambda_1}{\SExtn_X(\lambda_0)}, \frac{\lambda_1}{\SExtn_X(\lambda_1)}\right). \nonumber
	\end{gather}
	Using the homogeneity of the metric $d_F$ and (\ref{eq:C0}) we deduce
	\begin{equation}
	\label{eq:C2}
	d_F\left(\frac{\lambda_0}{\SExtn_X(\lambda_0)}, \frac{\lambda_1}{\SExtn_X(\lambda_0)}\right) = \frac{1}{\SExtn_X(\lambda_0)} \cdot d_F\left(\lambda_0, \lambda_1\right) \preceq_{F,\mathcal{K}} d_F\left(\lambda_0, \lambda_1\right).
	\end{equation}
	Directly from the definition of $d_F$ it follows that
	\begin{equation}
	\label{eq:C3}
	d_F\left(\frac{\lambda_1}{\SExtn_X(\lambda_0)}, \frac{\lambda_1}{\SExtn_X(\lambda_1)}\right) = \bigg|\frac{1}{\SExtn_X(\lambda_0)} - \frac{1}{\SExtn_X(\lambda_1)} \bigg| \cdot \|\lambda_1\|_F.
	\end{equation}
	Using Theorem \ref{theo:ext_lip} and (\ref{eq:C0}) we deduce
	\begin{equation}
	\label{eq:C4}
	\bigg|\frac{1}{\SExtn_X(\lambda_0)} - \frac{1}{\SExtn_X(\lambda_1)} \bigg| = \frac{| \SExtn_X(\lambda_0) - \SExtn_X(\lambda_1)|}{\SExtn_X(\lambda_0) \cdot \SExtn_X(\lambda_1)} \preceq_{F,\mathcal{K}} d_F(\lambda_0,\lambda_1).
	\end{equation}
	Putting together (\ref{eq:C1}), (\ref{eq:C2}), (\ref{eq:C3}), (\ref{eq:C4}), and using the bound $\|\lambda_1\|_F \preceq 1$, we conclude
	\[
	d_F\left(\frac{\lambda_0}{\SExtn_X(\lambda_0)}, \frac{\lambda_1}{\SExtn_X(\lambda_1)}\right) \preceq_{F,\mathcal{K}} d_F(\lambda_0,\lambda_1). \qedhere
	\]
\end{proof}

\section{Counting filling closed curves on surfaces}

\subsection*{Outline of this section.} In this section we give a complete proof of Theorem \ref{theo:main}, the main result of this paper, using the tracking method sketched in \S 1. We begin with a detailed outline of the proof of Theorem \ref{theo:main}. In accordance with the outlined strategy, we proceed to study certain localized versions of the counting function of interest. See Theorem \ref{theo:count_loc} for the main estimate concerning these local counting functions. A more precise version of Theorem \ref{theo:main} is then deduced as a direct consequence of this estimate. See Theorem \ref{theo:main_strong} for a precise statement. We end this section by stating a stronger version of Theorem \ref{theo:main} that applies to counting functions of non-connected filling closed multi-curves and which can also be proved using the tracking method.

\subsection*{Outline of the proof of Theorem \ref{theo:main}.} For the rest of this section we fix an integer $g \geq 2$ and a connected, oriented, closed surface $S_g$ of genus $g$. Unless otherwise stated we will not distinguish between closed curves on $S_g$ and their free homotopy classes. Recall that $\mcg$ denotes the mapping class group of $S_g$. Recall that $\mathcal{C}_g$ denotes the space of geodesic currents on $S_g$. Recall that $i(\cdot,\cdot)$ denotes the geometric intersection number pairing on $\mathcal{C}_g$. Recall that $\mathcal{C}_g^* \subseteq \mathcal{C}_g$ denotes the open subset of filling geodesic currents on $S_g$. Let $\alpha \in \mathcal{C}_g^*$ be a filling geodesic current on $S_g$ and $\beta$ be a filling closed curve on $S_g$. Recall that for every $L > 0$ we consider the counting function
\begin{equation}
\label{eq:c}
c(\alpha,\beta,L) := \#\{ \gamma \in \mcg \cdot \beta  \ | \  i(\alpha,\gamma) \leq L \}.
\end{equation}
Our goal is to prove a quantitative estimate with a power saving error term for this counting function. It will be convenient to consider the counting function defined for every $L > 0$ as
\begin{equation}
\label{eq:f}
f(\alpha,\beta,L) := \# \{\mc \in \mcg \ | \ i(\alpha,\mc.\beta) \leq L\}.
\end{equation}
As the closed curve $\beta$ on $S_g$ is filling, $\mathrm{Stab}(\beta) \subseteq \mcg$ is finite. In particular, for every $L > 0$,
\begin{equation}
\label{eq:count_cf}
f(\alpha,\beta,L) = |\mathrm{Stab}(\beta)| \cdot c(\alpha,\beta,L).
\end{equation}

It is equivalent then for our purposes to prove a quantitative estimate with a power saving error term for the counting function $f(\alpha,\beta,L)$. We prove this estimate following the tracking method sketched in \S 1. Roughly speaking, using the tracking principle in Theorem \ref{theo:track}, we approximate the counting function $f(\alpha,\beta,L)$ by finer and finer collections of bisector counting functions for mapping class group orbits of Teichmüller space, which we then estimate using  Theorem \ref{theo:bisect_count}. To give a more detailed account of this strategy, let us review some of the notation introduced in previous sections.

Recall that $\tt$ denotes the Teichmüller space of marked complex structures on $S_g$ and that $d_\mathcal{T}$ denotes the Teichmüller metric on $\mathcal{T}_g$. Recall that $\qut$ denotes the Teichmüller space of marked unit area quadratic differentials on $S_g$ and that $\pi \colon \qut \to \tt$ denotes the natural projection to $\mathcal{T}_g$. Recall that $\mf$ denotes the space of singular measured foliations on $S_g$ and that $\Re(q), \Im(q) \in \mf$ denote the vertical and horizontal foliations of $q \in \qut$. Recall that $\pmf $ denotes the space of projective singular measured foliations on $S_g$ and that $[\lambda] \in \pmf$ denotes the projective class of $\lambda \in \mf$.

Recall that $\Delta \subseteq \mathcal{T}_g \times \mathcal{T}_g$ denotes the diagonal of $\mathcal{T}_g \times \mathcal{T}_g$, that $S(X) := \pi^{-1}(X)$ for every $X \in \mathcal{T}_g$, and that $q_s, q_e \colon \mathcal{T}_g \times \mathcal{T}_g - \Delta \to \mathcal{Q}^1\mathcal{T}_g$ denote the maps which to every pair of distinct points $X \neq Y \in \mathcal{T}_g$ assign the marked quadratic differentials $q_s(X,Y) \in S(X)$ and $q_e(X,Y) \in S(Y)$ corresponding to the cotangent directions at $X$ and $Y$ of the unique Teichmüller geodesic segment from $X$ to $Y$.

Fix auxiliary marked hyperbolic structures $X,Y \in \mathcal{T}_g$. The tracking principle in Theorem \ref{theo:track} provides a constant $\kappa = \kappa(g) > 0$ depending only on $g$ such that for all but quantitively few mapping classes $\mc \in \mcg$, if $r := d_\mathcal{T}(X,\mc.Y)$, $q_s := q_s(X,\mc.Y)$, and $q_e := q_e(X,\mc.Y)$, then
\begin{equation}
\label{eq:track}
i(\alpha,\mc.\beta) = i(\alpha,\Re(q_s)) \cdot i(\mc.\beta,\Im(q_e)) \cdot e^r + O_{X,Y,K,\beta}\left(e^{(1-\kappa)r}\right).
\end{equation}
This estimate shows that the threshold $i(\alpha,\mc.\beta)$ of the counting function $f(\alpha,\beta,L)$ can be approximated by $e^r$ times a varying coefficient $i(\alpha,\Re(q_s)) \cdot i(\mc.\beta,\Im(q_e))$ depending on $\mc \in \mcg$. This coefficient can be controlled using bisector information. More concretely, if we restrict $[\Re(q_s)]$ and $[\Im(q_e)]$ to belong to small subsets of $\pmf$, the varying coefficient will be almost constant. 

After discarding quantitatively few mapping classes to ensure (\ref{eq:track}) holds, we refine the counting function $f(\alpha,\beta,L)$ into local counting functions with prescribed bisector information using a partition of $\pmf$ into simplices in Dehn-Thurston coordinates. Each one of these local counting functions can be estimated using Theorem \ref{theo:bisect_count}. How well the sum of the leading terms of these estimates approximates the original counting function $f(\alpha,\beta,L)$ depends on how small the variation of the leading coefficient in (\ref{eq:track}) is when restricting to each of the local counting functions. The magnitude of these variations can be controlled using Theorems \ref{theo:curr_lip} and \ref{theo:ext_lip_2}. The more we refine the partition of $\pmf$ the smaller these variations become. Each local counting function also contributes an error term as in Theorem \ref{theo:bisect_count}. To ensure the sum of these error terms remains controlled we refine the partition of $\pmf$ at an appropriate rate with respect to the parameter $L > 0$.

\subsection*{Local counting estimates.} As mentioned above, to prove Theorem \ref{theo:main} we study certain localizations of the counting function $f(\alpha,\beta,L)$. Recall that for every $X \in \mathcal{T}_g$ and every $U \subseteq \pmf$ we denote
\[
\mathrm{Sect}_U(X) := \{Y \in \mathcal{T}_g \setminus \{X\} \ | \ [\Re(q_s(X,Y))] \in U\}.
\]
Given a filling geodesic current $\alpha \in \mathcal{C}_g^*$, a filling closed curve $\beta$ on $S_g$, a pair of points $X,Y \in \mathcal{T}_g$, a pair of subsets $U,V \subseteq \pmf$, and $L > 0$, consider the local counting function
\[
f(\alpha,\beta,X,Y,U,V,L) := \# \left\lbrace
\begin{array}{l | l}
\mc \in \mcg \ & \ i(\alpha,\mc.\beta) \leq L, \\ 
& \ \mc.Y \in \mathrm{Sect}_U(X), \\
& \ \mc^{-1}.X \in \mathrm{Sect}_V(Y).
\end{array}
\right\rbrace.
\]
Our immediate goal is to prove a quantitative estimate with a power saving error term for this counting function. Theorems \ref{theo:bisect_count} and \ref{theo:track_mult} will be the main tools used in the proof of this estimate.

Let us introduce some notation that will help us simplify the statements of the results that follow. Recall that $\SExt_X(\lambda)$ denotes the square root of the extremal length of a singular measured foliation $\lambda \in \mf$ with respect to a marked complex structure $X \in \mathcal{T}_g$. Recall that for $X \in \mathcal{T}_g$ we denote by $E_X \colon \pmf \to \mf$ the map which to every projective class $[\lambda] \in \pmf$ assigns the unique singular measured foliation $E_X([\lambda]) := \lambda/\SExt_X(\lambda) \in \mf$ in the class $[\lambda]$ of unit extremal length with respect to $X$. Recall that $\pmf$ is compact. Given a filling geodesic current $\alpha \in \mathcal{C}_g^*$, a marked complex structure $X \in \mathcal{T}_g$, and a non-empty subset $U \subseteq \pmf$, consider the positive constants
\begin{gather*}
m(\alpha,X,U) := \inf_{[\lambda] \in U} i (\alpha, E_X([\lambda])),\\
M(\alpha,X,U) := \sup_{[\lambda] \in U} i (\alpha, E_X([\lambda])).
\end{gather*}

Let $X,Y \in \mathcal{T}_g$, $C > 0$, and $\kappa > 0$. Recall that, motivated by Theorem \ref{theo:teich_count}, we say that a subset of mapping classes $\M \subseteq \mcg$ is $(X,Y,C,\kappa)$-sparse if the following bound holds for every $R > 0$,
\[
\#\{\mc \in  \M \ | \ d_\mathcal{T}(X,\mc.Y) \leq R \} \leq C \cdot e^{(6g-6-\kappa)R}.
\] 

Recall that for every pair of points $X,Y \in \mathcal{T}_g$, every pair of subsets $U,V \subseteq \pmf$, and every $R > 0$, we consider the bisector counting function 
\[
N(X,Y,U,V,R) = \# \left\lbrace
\begin{array}{l | l}
\mc \in \mcg \ & \ d_\mathcal{T}(X,\mc.Y) \leq R, \\ 
& \ \mc.Y \in \mathrm{Sect}_U(X), \\
& \ \mc^{-1}.X \in \mathrm{Sect}_V(Y).
\end{array}
\right\rbrace.
\]

The following result gives an upper bound for the local counting function $f(\alpha,\beta,X,Y,U,V,L)$ in terms of the bisector counting function $N(X,Y,U,V,R)$. The tracking principle in Theorem \ref{theo:track_mult} is the main tool used in the proof of this result.

\begin{proposition} 
	\label{prop:up_2}
	There exists a constant $\kappa_3 = \kappa_3(g) > 0$ such that for every compact subset $\mathcal{K} \subseteq \mathcal{T}_g$, every compact subset $K \subseteq \mathcal{C}_g^*$, and every filling closed curve $\beta$ on $S_g$ there exist constants $A = A(\mathcal{K},K,\beta) > 0$ and  $L_1 = L_1(\mathcal{K},K,\beta) > 0$ such that for every $\alpha \in K$, every $X,Y \in \mathcal{K}$, every $U,V \subseteq \pmf$ non-empty, and every $L > L_1$, if 
	\[
	R:= \log(m(\alpha,X,U)^{-1} \cdot m(\beta,Y,V)^{-1} \cdot (1-A \cdot L^{-\kappa_3})^{-1} \cdot L),
	\]
	then the following bound holds,
	\[
	f(\alpha,\beta,X,Y,U,V,L) - N(X,Y,U,V,R) \preceq_{\mathcal{K},K,\beta} L^{6g-6-\kappa_3}.
	\]
\end{proposition}

\begin{proof}
	Let $\kappa_2 = \kappa_2(g) > 0$ be as in Theorem \ref{theo:track_mult} and $\kappa_3 = \kappa_3(g) := \min\{\kappa_2/2,3g-3\} > 0$. Fix $\mathcal{K} \subseteq \mathcal{T}_g$ compact, $K \subseteq \mathcal{C}_g^*$ compact, and $\beta$ a filling closed curve on $S_g$. Let $C = C(\mathcal{K},\beta) > 0$ and $A = A(\mathcal{K},K,\beta) > 0$ be as in Theorem \ref{theo:track_mult}. Fix $X,Y \in \mathcal{K}$. Let $M := M(X,Y,\beta) \subseteq \mcg$ be the $(X,Y,C,\kappa_2)$-sparse subset of mapping classes provided by Theorem \ref{theo:track_mult}. Fix $\alpha \in K$. Theorem \ref{theo:track_mult} ensures that for every $\mc \in \mcg \setminus M$, if $r := d_\mathcal{T}(X,\mc.Y)$, $q_s := q_s(X,\mc.Y)$, and $q_e := q_e(X,\mc.Y)$, then
	\begin{gather}
		\label{eq:t_0}
		i(\alpha,\Re(q_s)) \cdot i(\mc.\beta,\Im(q_e)) \cdot e^r \cdot (1 - A \cdot e^{-\kappa_2 r}) \leq i(\alpha,\mc.\beta).
	\end{gather}

	Notice that for every mapping class $\mc \in \mcg$ such that $\mc.Y \neq X$,
	\[
	q_e(X,\mc.Y) = \mc.q_e(\mc^{-1}.X,Y)= -\mc.q_s(Y,\mc^{-1}.X).
	\]
	As multiplication by $-1$ exchanges the vertical and horizontal foliations of quadratic differentials, this implies that for every mapping class $\mc \in \mcg$ such that $\mc.Y \neq X$,
	\[
	\Im(q_e(X,\mc.Y)) = \Im(-\mc.q_s(Y,\mc^{-1}.X)) = \mc.\Re(q_s(Y,\mc^{-1}.X)).
	\]
	In particular, for every mapping class $\mc \in \mcg$ such that $\mc.Y \neq X$,
	\begin{equation}
	\label{eq:t00}
	i(\mc.\beta,\Im(q_e(X,\mc.Y))) = i(\mc.\beta,\mc.\Re(q_s(Y,\mc^{-1}.X))) = i(\beta,\Re(q_s(Y,\mc^{-1}.X))).
	\end{equation}
	Recall that $\SExt_{\pi(q)}(\Re(q)) = 1$ for every $q \in \qut$. From this fact and (\ref{eq:t00}) we deduce the following identities hold for every $\mc \in \mcg$ such that $\mc.Y \neq X$,
	\begin{gather}
	i(\alpha,\Re(q_s(X,\mc.Y))) = i(\alpha,E_X([\Re(q_s(X,\mc.Y))])), \label{eq:ta}\\
	i(\mc.\beta,\Im(q_e(X,\mc.Y))) = i(\beta,E_Y([\Re(q_s(Y,\mc^{-1}.X))])). \label{eq:tb}
	\end{gather}
	
	Fix $U,V \subseteq \pmf$ non-empty and $L > 1$. Consider the truncated local counting function
	\[
	f^\dagger(\alpha,\beta,X,Y,U,V,L) := \# \left\lbrace
	\begin{array}{l | l}
	\mc \in \mcg \setminus M \ & \ i(\alpha,\mc.\beta) \leq L, \\ 
	& \ \mc.Y \in \mathrm{Sect}_U(X), \\
	& \ \mc^{-1}.X \in \mathrm{Sect}_V(Y).
	\end{array}
	\right\rbrace.
	\]
	By Proposition \ref{prop:coarse}, there exists a constant $B = B(\mathcal{K},K,\beta) > 0$ such that for every $\mc \in \mcg$,
	\begin{equation}
	\label{eq:t1}
	d_\mathcal{T}(X,\mc.Y) \leq \log( B \cdot i(\alpha,\mc.\beta)).
	\end{equation}
	As $M \subseteq \mcg$ is $(X,Y,C,\kappa_2)$-sparse, (\ref{eq:t1}) implies 
	\[
	\#\{\mc \in M \ | \ i(\alpha,\mc.\beta) \leq L\} \leq C \cdot B^{6g-6-\kappa_2} \cdot L^{6g-6-\kappa_2} \preceq_{\mathcal{K},K,\beta} L^{6g-6-\kappa_3}.
	\]
	In particular, the following estimate holds,
	\begin{equation}
	\label{eq:t2}
	f(\alpha,\beta,X,Y,U,V,L) = f^\dagger(\alpha,\beta,X,Y,U,V,L) + O_{\mathcal{K},K,\beta}\left( L^{6g-6-\kappa_3}\right).
	\end{equation}
	
	Consider now the further truncated local counting function
	\[
	f^{\ddagger}(\alpha,\beta,X,Y,U,V,L) := \# \left\lbrace
	\begin{array}{l | l}
	\mc \in \mcg \setminus M \ & \ i(\alpha,\mc.\beta) \leq L, \\ 
	& \ d_\mathcal{T}(X,\mc.Y) \geq \log(L^{1/2}), \\
	& \ \mc.Y \in \mathrm{Sect}_U(X), \\
	& \ \mc^{-1}.X \in \mathrm{Sect}_V(Y).
	\end{array}
	\right\rbrace.
	\]
	Theorem \ref{theo:teich_count} guarantees that
	\[
	\# \{\mc \in \mcg \ | \ d_\mathcal{T}(X,\mc.Y) \leq \log(L^{1/2}) \} \preceq_{\mathcal{K}} L^{3g-3} \preceq_{\mathcal{K}} L^{6g-6-\kappa_3}.
	\]
	In particular, the following estimate holds,
	\begin{equation}
	\label{eq:t3}
	f^{\dagger}(\alpha,\beta,X,Y,U,V,L) = f^{\ddagger}(\alpha,\beta,X,Y,U,V,L) + O_\mathcal{K}\left(L^{6g-6-\kappa_3}\right).
	\end{equation}
	
	Let $L_1 = L_1(\mathcal{K},K,L) := \max\{(2A)^{1/\kappa_3},1\} > 0$. For the rest of this proof we assume $L > L_1$. This guarantees $1- A \cdot L^{-\kappa_3} > 1/2$. Consider the parameter
	\[
	R:= \log(m(\alpha,X,U)^{-1} \cdot m(\beta,Y,V)^{-1} \cdot (1-A \cdot L^{-\kappa_3})^{-1} \cdot L).
	\]
	The bound in (\ref{eq:t_0}) together with the identities in (\ref{eq:ta}) and (\ref{eq:tb}) ensure that every mapping class $\mc \in \mcg \setminus M$ such that $i(\alpha,\mc.\beta) \leq L$, $d_\mathcal{T}(X,\mc.Y) \geq \log(L^{1/2})$, $\mc.Y \in \mathrm{Sect}_U(X)$, and $\mc^{-1}.X \in \mathrm{Sect}_V(Y)$ must also satisfy $d_\mathcal{T}(X,\mc.Y) \leq R$. In particular, the following bound holds,
	\begin{equation}
	\label{eq:t4}
	 f^{\ddagger}(\alpha,\beta,X,Y,U,V,L) \leq  N(X,Y,U,V,R).
	\end{equation}

	Putting together (\ref{eq:t2}), (\ref{eq:t3}), and (\ref{eq:t4}), we conclude
	\[
	f(\alpha,\beta,X,Y,U,V,L) - N(X,Y,U,V,R) \preceq_{\mathcal{K},K,\beta} L^{6g-6-\kappa_3}. \qedhere 
	\]
\end{proof}

The following result gives a lower bound for the local counting function $f(\alpha,\beta,X,Y,U,V,L)$ in terms of the bisector counting function $N(X,Y,U,V,R)$. As in the case of Proposition \ref{prop:up_2}, the tracking principle in Theorem \ref{theo:track_mult} is the main tool used in the proof of this result.

\begin{proposition} 
	\label{prop:low_2}
	There exists a constant $\kappa_3 = \kappa_3(g) > 0$ such that for every $\mathcal{K} \subseteq \mathcal{T}_g$ compact, every $K \subseteq \mathcal{C}_g^*$ compact, and every filling closed curve $\beta$ on $S_g$ there exists a  constant $A = A(\mathcal{K},K,\beta) > 0$ such that for every $\alpha \in K$, every $X,Y \in \mathcal{K}$, every $U,V \subseteq \pmf$ non-empty, and every $L > 1$, if 
	\[
	R':= \log(M(\alpha,X,U)^{-1} \cdot M(\beta,Y,V)^{-1} \cdot (1+A \cdot L^{-\kappa_3})^{-1} \cdot L),
	\]
	then the following bound holds,
	\[
	N(X,Y,U,V,R') - f(\alpha,\beta,X,Y,U,V,L) \preceq_{\mathcal{K},K,\beta} L^{6g-6-\kappa_3}.
	\]
\end{proposition}

\begin{proof}
	Let $\kappa_2 = \kappa_2(g) > 0$ be as in Theorem \ref{theo:track_mult} and $\kappa_3 = \kappa_3(g) := \min\{\kappa_2/2,3g-3\} > 0$. Fix $\mathcal{K} \subseteq \mathcal{T}_g$ compact, $K \subseteq \mathcal{C}_g^*$ compact, and $\beta$ a filling closed curve on $S_g$. Let $C = C(\mathcal{K},\beta) > 0$ and $A = A(\mathcal{K},K,\beta) > 0$ be as in Theorem \ref{theo:track_mult}. Fix $X,Y \in \mathcal{K}$. Let $M := M(X,Y,\beta) \subseteq \mcg$ be the $(X,Y,C,\kappa_2)$-sparse subset of mapping classes provided by Theorem \ref{theo:track_mult}. Fix $\alpha \in K$. Theorem \ref{theo:track_mult} ensures that for every $\mc \in \mcg \setminus M$, if $r := d_\mathcal{T}(X,\mc.Y)$, $q_s := q_s(X,\mc.Y)$, and $q_e := q_e(X,\mc.Y)$, then
	\begin{gather}
	\label{eq:r_0}
	i(\alpha,\mc.\beta) \leq i(\alpha,\Re(q_s)) \cdot i(\mc.\beta,\Im(q_e)) \cdot e^r \cdot (1 + A \cdot e^{-\kappa_2 r}).
	\end{gather}
	
	As explained in the proof of Proposition \ref{prop:up_2}, see (\ref{eq:ta}) and (\ref{eq:tb}), the following identities hold for every mapping class $\mc \in \mcg$ such that $\mc.Y \neq X$,
	\begin{gather}
	i(\alpha,\Re(q_s(X,\mc.Y))) = i(\alpha,E_X([\Re(q_s(X,\mc.Y))])), \label{eq:ra}\\
	i(\mc.\beta,\Im(q_e(X,\mc.Y))) = i(\beta,E_Y([\Re(q_s(Y,\mc^{-1}.X))])). \label{eq:rb}
	\end{gather}
	
	Fix $U,V \subseteq \pmf$ non-empty and $L > 1$. Consider the parameter
	\[
	R':= \log(M(\alpha,X,U)^{-1} \cdot M(\beta,Y,V)^{-1} \cdot (1+A \cdot L^{-\kappa_3})^{-1} \cdot L).
	\]
	Consider the truncated bisector counting function
	\[
	N^\dagger(X,Y,U,V,R') := \# \left\lbrace
	\begin{array}{l | l}
	\mc \in \mcg \setminus M \ & \ d_\mathcal{T}(X,\mc.Y) \leq R', \\ 
	& \ \mc.Y \in \mathrm{Sect}_U(X), \\
	& \ \mc^{-1}.X \in \mathrm{Sect}_V(Y).
	\end{array}
	\right\rbrace.
	\]
	The compactness of $K \subseteq \mathcal{C}_g^*$, $\mathcal{K} \subseteq \mathcal{T}_g$, and $\pmf$, and the assumption that $\beta$ is filling ensure that
	\begin{gather}
	M(\alpha,X,U) \geq \inf_{[\lambda] \in \pmf} i (\alpha, E_X([\lambda])) \succeq_{\mathcal{K},K} 1, \label{eq:r1}\\
	M(\beta,Y,V) \geq \inf_{[\lambda] \in \pmf} i (\beta, E_Y([\lambda])) \succeq_{\mathcal{K},\beta} 1. \label{eq:r2}
	\end{gather}
	As $M \subseteq \mcg$ is $(X,Y,C,\kappa_2)$-sparse, (\ref{eq:r1}) and (\ref{eq:r2}) imply
	\[
	\# \{\mc \in M \ | \ d_\mathcal{T}(X,\mc.Y) \leq R'\} \preceq_{\mathcal{K},K,\beta} L^{6g-6-\kappa_2} \preceq_{\mathcal{K},K,\beta} L^{6g-6-\kappa_3}.
	\]
	In particular, the following estimate holds,
	\begin{equation}
	\label{eq:r3}
	N(X,Y,U,V,R') = N^\dagger(X,Y,U,V,R') + O_{\mathcal{K},K,\beta}\left(L^{6g-6-\kappa_3}\right).
	\end{equation}
	
	Consider now the further truncated bisector counting function
	\[
	N^\ddagger(X,Y,U,V,R') := \# \left\lbrace
	\begin{array}{l | l}
	\mc \in \mcg \setminus M \ & \ d_\mathcal{T}(X,\mc.Y) \leq R', \\ 
	& \ d_\mathcal{T}(X,\mc.Y) \geq \log(L^{1/2}), \\
	& \ \mc.Y \in \mathrm{Sect}_U(X), \\
	& \ \mc^{-1}.X \in \mathrm{Sect}_V(Y).
	\end{array}
	\right\rbrace.
	\]
	Theorem \ref{theo:teich_count} guarantees that
	\[
	\# \{\mc \in \mcg \ | \ d_\mathcal{T}(X,\mc.Y) \leq \log(L^{1/2}) \} \preceq_{\mathcal{K}} L^{3g-3} \preceq_{\mathcal{K}} L^{6g-6-\kappa_3}.
	\]
	In particular, the following estimate holds,
	\begin{equation}
	\label{eq:r4}
	N^\dagger(X,Y,U,V,R') = N^\ddagger(X,Y,U,V,R') + O_{\mathcal{K}}\left( L^{6g-6-\kappa_3}\right).
	\end{equation}
	
	The bound in (\ref{eq:r_0}) together with the identities in (\ref{eq:ra}) and (\ref{eq:rb}) ensure that every mapping class $\mc \in \mcg \setminus M$ such that $d_\mathcal{T}(\alpha,\mc.\beta) \leq R'$, $d_\mathcal{T}(X,\mc.Y) \geq \log(L^{1/2})$, $\mc.Y \in \mathrm{Sect}_U(X)$, and $\mc^{-1}.X \in \mathrm{Sect}_V(Y)$ must also satisfy $i(\alpha,\beta) \leq L$. In particular, the following bound holds,
	\begin{equation}
	\label{eq:r5}
	N^\ddagger(X,Y,U,V,R') \leq  f(\alpha,\beta,X,Y,U,V,L).
	\end{equation}
	
	Putting together (\ref{eq:r3}), (\ref{eq:r4}), and (\ref{eq:r5}), we conclude
	\[
	N(X,Y,U,V,R') - f(\alpha,\beta,X,Y,U,V,L) \preceq_{\mathcal{K},K,\beta} L^{6g-6-\kappa_3}. \qedhere
	\]
\end{proof}

Putting together Propositions \ref{prop:up_2} and \ref{prop:low_2} we obtain the following bounds for the local counting function $f(\alpha,\beta,X,Y,U,V,L)$ in terms of the bisector counting function $N(X,Y,U,V,R)$. 

\begin{proposition} 
	\label{prop:total_2}
	There exists a constant $\kappa_3 = \kappa_3(g) > 0$ such that for every compact subset $\mathcal{K} \subseteq \mathcal{T}_g$, every compact subset $K \subseteq \mathcal{C}_g^*$, and every filling closed curve $\beta$ on $S_g$ there exist constants $A = A(\mathcal{K},K,\beta) > 0$ and  $L_1 = L_1(\mathcal{K},K,\beta) > 0$ such that for every $\alpha \in K$, every $X,Y \in \mathcal{K}$, every $U,V \subseteq \pmf$ non-empty, and every $L > L_1$, if 
	\begin{gather*}
	R:= \log(m(\alpha,X,U)^{-1} \cdot m(\beta,Y,V)^{-1} \cdot (1-A \cdot L^{-\kappa_3})^{-1} \cdot L),\\
	R' := \log(M(\alpha,X,U)^{-1} \cdot M(\beta,Y,V)^{-1} \cdot (1+A \cdot L^{-\kappa_3})^{-1} \cdot L),
	\end{gather*}
	then the following bounds hold,
	\begin{gather*}
	f(\alpha,\beta,X,Y,U,V,L) - N(X,Y,U,V,R) \preceq_{\mathcal{K},K,\beta} L^{6g-6-\kappa_3},\\
	N(X,Y,U,V,R') - f(\alpha,\beta,X,Y,U,V,L) \preceq_{\mathcal{K},K,\beta} L^{6g-6-\kappa_3}.
	\end{gather*}
\end{proposition}

The bounds in Proposition \ref{prop:total_2} for the local counting function $f(\alpha,\beta,X,Y,U,V,L)$ can be further developed using the effective estimate in Theorem \ref{theo:bisect_count} for the bisector counting function $N(X,Y,U,V,R)$ under the assumption that $U,V \subseteq \pmf$ are simplices in Dehn-Thurston coordinates. To explore this idea further, let us review some of the relevant notation introduced in previous sections.

Recall that $\Sigma$ denotes the piecewise linear manifold $\Sigma := \mathbf{R}^2 / \langle-1\rangle$ endowed with the quotient Euclidean metric. Recall that the product $\IT^{3g-3}$ is a piecewise linear manifold which we endow with the $L^2$ product metric.  Recall that the projectivization $P\Sigma^{3g-3}$ of $\smash{\Sigma^{3g-3}}$ can be canonically identified with the subset $\smash{\Sigma_u^{3g-3}} \subseteq \smash{\Sigma^{3g-3}}$ of points in $\Sigma^{3g-3}$ of unit $L^1$-norm through a homeomorphism we denote by $u \colon \pmf \to \smash{\Sigma_u^{3g-3}}$. Recall that we endow the subset $\smash{\Sigma_u^{3g-3}} \subseteq \Sigma^{3g-3}$ with the induced Euclidean metric. Recall that the subset $\smash{\Sigma_u^{3g-3}} \subseteq \Sigma^{3g-3}$ can be represented as a union of $2^{3g-3}$ closed affine simplices of dimension $6g-7$ we refer to as the facets of $\smash{\Sigma_u^{3g-3}}$.

Recall that any set of Dehn-Thurston coodinates provides a homeomorphism $F \colon \mf \to \IT^{3g-3}$ equivariant with respect to the natural $\mathbf{R}^+$ scaling actions on $\mf$ and $\IT^{3g-3}$. Recall that we denote the projectivization of any such homeomorphism by $PF \colon \pmf \to P\Sigma^{3g-3}$. Recall that a measurable subset $U \subseteq \pmf$ is said to be an $F$-simplex if there exists an open affine simplex $W \subseteq \smash{\Sigma_u^{3g-3}}$ of dimension $6g-7$ contained in an $F$-facet of $\smash{\Sigma_u^{3g-3}}$ such that $W \subseteq u \circ PF(U) \subseteq \smash{\overline{W}}$.

Recall the definition of the constant $b_g > 0$ introduced in (\ref{eq:bg}). Recall that $\nu_\mathrm{Thu}$ denotes the Thurston measure on $\mf$. Recall that for every $X \in \mathcal{T}_g$ we consider the measure $\nu_X$ on $\pmf$ which to every measurable subset $U \subseteq \pmf$ assigns the value
\[
\nu_X(U) := \nu_{\mathrm{Thu}}\left(\left\lbrace\lambda \in \mf \ | \ \SExt_X(\lambda) \leq 1, \ [\lambda] \in U \right\rbrace \right).
\]

Combining Proposition \ref{prop:total_2} with Theorem \ref{theo:bisect_count} we obtain the following bounds for the local counting function $f(\alpha,\beta,X,Y,U,V,L)$ independent of the bisector counting function $N(X,Y,U,V,R)$.

\begin{proposition} 
	\label{prop:total_DT}
	There exists $\kappa_4 = \kappa_4(g) > 0$ such that for every $\mathcal{K} \subseteq \mathcal{T}_g$ compact, every $K \subseteq \mathcal{C}_g^*$ compact, and every filling closed curve $\beta$ on $S_g$ there exist constants $A = A(\mathcal{K},K,\beta) > 0$ and  $L_2 = L_2(\mathcal{K},K,\beta) > 0$ such that for every $\alpha \in K$, every $X,Y \in \mathcal{K}$, every set of Dehn-Thurston coordinates $F \colon \mf \to \smash{\Sigma^{3g-3}}$, every pair of non-empty $F$-simplices $U,V \subseteq \pmf$, and every $L > L_2$, if 
	\begin{gather*}
	C:= m(\alpha,X,U) \cdot m(\beta,Y,V) \cdot (1-A \cdot L^{-\kappa_4}),\\
	C':= M(\alpha,X,U) \cdot M(\beta,Y,V) \cdot (1+A \cdot L^{-\kappa_4}),
	\end{gather*}
	then the following bounds hold,
	\begin{gather*}
	f(\alpha,\beta,X,Y,U,V,L) - \frac{\nu_X(U) \cdot \nu_Y(V)}{C^{6g-6} \cdot b_g} \cdot L^{6g-6} \preceq_{\mathcal{K},K,\beta,F} L^{6g-6-\kappa_4},\\
	\frac{\nu_X(U) \cdot \nu_Y(V)}{(C')^{6g-6} \cdot b_g} \cdot L^{6g-6}  - f(\alpha,\beta,X,Y,U,V,L) \preceq_{\mathcal{K},K,\beta,F} L^{6g-6-\kappa_4}.
	\end{gather*}
\end{proposition}

\begin{proof}
	Let $\kappa_1 = \kappa_1(g) > 0$ be as in Theorem \ref{theo:bisect_count}, $\kappa_3 = \kappa_3(g) > 0$ be as in Proposition \ref{prop:total_2}, and $\kappa_4 = \kappa_4(g) := \min\{\kappa_1,\kappa_3\} > 0$. Fix a compact subset $\mathcal{K} \subseteq \mathcal{T}_g$, a compact subset $K \subseteq \mathcal{C}_g^*$, and a filling closed curve $\beta$ on $S_g$. Let $A = A(\mathcal{K},K,\beta) > 0$ and $L_1 = L_1(\mathcal{K},K,\beta) > 0$ be as in Proposition \ref{prop:total_2}. Consider the constant $L_2 = L_2(\mathcal{K},K,\beta) := \max\{L_0,(2A)^{1/\kappa_4},1\} > 0$. Fix $\alpha \in K$ and $X,Y \in \mathcal{K}$. Let $F \colon \mf \to \Sigma^{3g-3}$ be a set of Dehn-Thurston coordinates of $\mf$. Consider a pair of non-empty $F$-simplices $U,V \subseteq \pmf$. Let $L > L_2$. Notice that $1-A \cdot L^{-\kappa_4} > 1/2$. Consider the parameters
	\begin{gather*}
	R:= \log(m(\alpha,X,U)^{-1} \cdot m(\beta,Y,V)^{-1} \cdot (1-A \cdot L^{-\kappa_4})^{-1} \cdot L),\\
	R' := \log(M(\alpha,X,U)^{-1} \cdot M(\beta,Y,V)^{-1} \cdot (1+A \cdot L^{-\kappa_4})^{-1} \cdot L).
	\end{gather*}
	Proposition \ref{prop:total_2} ensures that
	\begin{gather}
	f(\alpha,\beta,X,Y,U,V,L) - N(X,Y,U,V,R) \preceq_{\mathcal{K},K,\beta} L^{6g-6-\kappa_4}, \label{eq:y1} \\
	N(X,Y,U,V,R') - f(\alpha,\beta,X,Y,U,V,L) \preceq_{\mathcal{K},K,\beta} L^{6g-6-\kappa_4}. \label{eq:y2}
	\end{gather}
	
	Applying Theorem \ref{theo:bisect_count} we deduce
	\begin{gather}
	N(X,Y,U,V,R) = \frac{\nu_X(U) \cdot \nu_Y(V)}{b_g} \cdot e^{(6g-6)R} + O_{\mathcal{K},F}\left(e^{(6g-6-\kappa_4)R}\right), \label{eq:y3}\\
	N(X,Y,U,V,R') = \frac{\nu_X(U) \cdot \nu_Y(V)}{b_g} \cdot e^{(6g-6)R'} + O_{\mathcal{K},F}\left(e^{(6g-6-\kappa_4)R'}\right). \label{eq:y4}
	\end{gather}
	Consider the parameters
	\begin{gather*}
	C:= m(\alpha,X,U) \cdot m(\beta,Y,V) \cdot (1-A \cdot L^{-\kappa_4}),\\
	C':= M(\alpha,X,U) \cdot M(\beta,Y,V) \cdot (1+A \cdot L^{-\kappa_4}).
	\end{gather*}
	A direct computation shows that
	\begin{gather}
	e^{(6g-6)R} = C^{-(6g-6)} \cdot L^{6g-6}, \label{eq:y5}\\
	e^{(6g-6)R'} = (C')^{-(6g-6)} \cdot L^{6g-6}. \label{eq:y6}
	\end{gather}
	The compactness of $K \subseteq \mathcal{C}_g^*$, $\mathcal{K} \subseteq \mathcal{T}_g$, and $\pmf$, and the assumption that $\beta$ is filling ensure that
	\begin{gather*}
	M(\alpha,X,U) \geq m(\alpha,X,U) \geq \inf_{[\lambda] \in \pmf} i (\alpha, E_X([\lambda])) \succeq_{\mathcal{K},K} 1, \\
	M(\beta,Y,V) \geq m(\beta,Y,V) \geq \inf_{[\lambda] \in \pmf} i (\beta, E_Y([\lambda])) \succeq_{\mathcal{K},\beta} 1. 
	\end{gather*}
	These bounds together with the observation $1 + A \cdot L^{-\kappa_4} > 1-A \cdot L^{-\kappa_4} > 1/2$ imply
	\begin{gather}
	e^{(6g-6-\kappa_4)R} \preceq_{\mathcal{K},K,\beta} L^{6g-6-\kappa_4}, \label{eq:y7}\\
	e^{(6g-6-\kappa_4)R'} \preceq_{\mathcal{K},K,\beta} L^{6g-6-\kappa_4}. \label{eq:y8}
	\end{gather}
	
	Putting together (\ref{eq:y1}), (\ref{eq:y2}), (\ref{eq:y3}), (\ref{eq:y4}), (\ref{eq:y5}), (\ref{eq:y6}), (\ref{eq:y7}), and (\ref{eq:y8}), we conclude
	\begin{gather*}
	f(\alpha,\beta,X,Y,U,V,L) - \frac{\nu_X(U) \cdot \nu_Y(V)}{C^{6g-6} \cdot b_g} \cdot L^{6g-6} \preceq_{\mathcal{K},K,\beta,F} L^{6g-6-\kappa_4},\\
	\frac{\nu_X(U) \cdot \nu_Y(V)}{(C')^{6g-6} \cdot b_g} \cdot L^{6g-6}  - f(\alpha,\beta,X,Y,U,V,L) \preceq_{\mathcal{K},K,\beta,F} L^{6g-6-\kappa_4}. \qedhere
	\end{gather*}
\end{proof}

To obtain an effective estimate for the local counting function $f(\alpha,\beta,X,Y,U,V,L)$ from the bounds in Proposition \ref{prop:total_DT} we need to control the difference between the leading terms of these bounds. Given a set of Dehn-Thurston coordinates $F \colon \mf \to \Sigma^{3g-3}$ and a subset $U \subseteq \pmf$, denote by $\mathrm{diam}_F(U)$ the diameter of $u \circ PF (U) \subseteq \smash{\Sigma_u^{3g-3}}$ with respect to the induced Euclidean metric on $\smash{\Sigma_u^{3g-3}} \subseteq \smash{\Sigma^{3g-3}}$. The following estimate is a direct consequence of Theorems \ref{theo:curr_lip} and \ref{theo:ext_lip_2}.

\begin{proposition}
	\label{prop:bound_1}
	Let $K \subseteq \mathcal{C}_g$ compact, $\mathcal{K} \subseteq \mathcal{T}_g$ compact, and $F \colon \mf \to \Sigma^{3g-3}$ be a set of Dehn-Thurston coordinates. Then, for every $\alpha \in K$, every $X \in \mathcal{K}$, and every $U \subseteq \pmf$ non-empty,
	\[
	M(\alpha,X,U) - m(\alpha,X,U) \preceq_{K,\mathcal{K},F} \mathrm{diam}_F(U).
	\]
\end{proposition}

\begin{proof}
	Fix $\alpha \in K$, $X \in \mathcal{K}$, and $U \subseteq \pmf$ non-empy. Recall that
	\begin{gather*}
	m(\alpha,X,U) := \inf_{[\lambda] \in U} i (\alpha, E_X([\lambda])),\\
	M(\alpha,X,U) := \sup_{[\lambda] \in U} i (\alpha, E_X([\lambda])).
	\end{gather*}
	Notice that the function $i(\alpha,E_X(\cdot)) \colon \pmf \to \mathbf{R}^+$ can be written as the composition of the maps
	\begin{gather}
	i(\alpha,\cdot) \circ F^{-1} \colon \Sigma^{3g-3} \to \mathbf{R}^+, \label{eq:n1}\\
	F \circ E_X \circ PF^{-1} \circ u^{-1} \colon \smash{\Sigma_u^{3g-3}} \to \Sigma^{3g-3}, \label{eq:n2}\\
	u \circ PF \colon \pmf \to \smash{\Sigma_u^{3g-3}}. \label{eq:n3}
	\end{gather}
	Theorems \ref{theo:curr_lip} and \ref{theo:ext_lip_2} guarantee the maps in (\ref{eq:n1}) and (\ref{eq:n2}) are Lipschitz with Lipschitz constants bounded uniformly in terms of $K \subseteq \mathcal{C}_g$, $\mathcal{K} \subseteq \mathcal{T}_g$, and $F \colon \mf \to \Sigma^{3g-3}$. As $\mathrm{diam}_F(U)$ is by definition the diameter of the image of $U \subseteq \pmf$ under the map in (\ref{eq:n3}), we conclude
	\[
	M(\alpha,X,U) - m(\alpha,X,U) \preceq_{K,\mathcal{K},F} \mathrm{diam}_F(U). \qedhere
	\]
\end{proof}

Given a filling geodesic current $\alpha \in \mathcal{C}_g^*$ denote by $\nu_\alpha$ the measure on $\pmf$ which to every measurable subset $U \subseteq \pmf$ assigns the value
\[
\nu_\alpha(U) := \nu_{\mathrm{Thu}}\left(\left\lbrace\lambda \in \mf \ | \ i(\alpha,\lambda) \leq 1, \ [\lambda] \in U \right\rbrace \right).
\]
In the ensuing discussion we make use of the following scaling property of the Thurston measure: for every measurable subset $A \subseteq \mf$ and every $t > 0$,
\begin{equation}
\label{eq:scale_prop}
\nu_{\mathrm{Thu}}(t \cdot A) = t^{6g-6} \cdot \nu_{\mathrm{Thu}}(A).
\end{equation}

We now use the bound in Proposition \ref{prop:bound_1} to control the difference between the leading terms of the bounds in Proposition \ref{prop:total_DT} for the local counting function $f(\alpha,\beta,X,Y,U,V,L)$.

\begin{proposition}
	\label{prop:bound_2}
	Let $K \subseteq \mathcal{C}_g^*$ be a compact subset, $\mathcal{K} \subseteq \mathcal{T}_g$ be a compact subset, $\beta$ be a filling closed curve on $S_g$, $F \colon \mf \to \smash{\Sigma^{3g-3}}$ be a set of Dehn-Thurston coordinates, $A > 0$, $\kappa > 0$, and $L_3 = L_3(A,\kappa) := \smash{(2A)^{1/\kappa}} > 0$. Then, for every $\alpha \in K$, every $X,Y \in \mathcal{K}$, every pair of non-empty subsets $U,V \subseteq \pmf$, and every $L > L_3$, if 
	\begin{gather*}
	S := \nu_X(U) \cdot \nu_Y(V) \cdot \left(m(\alpha,X,U) \cdot m(\beta,Y,V) \cdot (1-A \cdot L^{-\kappa})\right)^{-(6g-6)},\\
	S' := \nu_X(U) \cdot \nu_Y(V) \cdot \left(M(\alpha,X,U) \cdot M(\beta,Y,V) \cdot (1+A \cdot L^{-\kappa})\right)^{-(6g-6)},
	\end{gather*}
	then the following bounds hold,
	\begin{gather}
	S' \leq \nu_\alpha(U) \cdot \nu_\beta(V) \leq S, \label{eq:m1}\\
	S - S' \preceq_{\mathcal{K},K,\beta,F} \nu_X(U) \cdot \nu_Y(V) \cdot\left(\mathrm{diam}_F(U) + \mathrm{diam}_F(V)\right) + A \cdot L^{-\kappa}. \label{eq:m2}
	\end{gather}
\end{proposition}

\begin{proof}
	Let $K \subseteq \mathcal{C}_g^*$ compact, $\mathcal{K} \subseteq \mathcal{T}_g$ compact, $\beta$ a filling closed curve on $S_g$, $F \colon \mf \to \smash{\Sigma^{3g-3}}$ a set of Dehn-Thurston coordinates, $A > 0$, $\kappa > 0$, and $L_3 = L_3(A,\kappa) := \smash{(2A)^{1/\kappa}} > 0$. Fix $\alpha \in K$, $X,Y \in \mathcal{K}$, non-empty subsets $U,V \subseteq \pmf$, and $L > L_3$. Notice that $1-A \cdot L^{-\kappa}  > 1/2$. Consider 
	\begin{gather}
		S := \nu_X(U) \cdot \nu_Y(V) \cdot \left(m(\alpha,X,U) \cdot m(\beta,Y,V) \cdot (1-A \cdot L^{-\kappa})\right)^{-(6g-6)}, \label{eq:ma}\\
		S' := \nu_X(U) \cdot \nu_Y(V) \cdot \left(M(\alpha,X,U) \cdot M(\beta,Y,V) \cdot (1+A \cdot L^{-\kappa})\right)^{-(6g-6)}. \label{eq:mb}
	\end{gather}
	
	We first show (\ref{eq:m1}) holds. A direct computation using the scaling property in (\ref{eq:scale_prop})  shows that
	\begin{gather*}
	M(\alpha,X,U)^{-(6g-6)} \cdot \nu_X(U) \leq \nu_\alpha(U) \leq m(\alpha,X,U)^{-(6g-6)} \cdot \nu_X(U),\\
	M(\beta,Y,V)^{-(6g-6)} \cdot \nu_Y(V) \leq \nu_\beta(V) \leq m(\beta,Y,V)^{-(6g-6)} \cdot \nu_Y(V).
	\end{gather*}
	Using these bounds and the observation $(1+A \cdot L^{-\kappa})^{-(6g-6)} < 1 < (1-A \cdot L^{-\kappa})^{-(6g-6)}$ we conclude
	\[
	S' \leq \nu_\alpha(U) \cdot \nu_\beta(V) \leq S.
	\]
	
	We now show (\ref{eq:m2}) holds. For convenience we denote
	\begin{align*}
	\hspace{+4.3cm} &m = m(\alpha,X,U), & &m' = m(\beta,Y,V), \hspace{+4.2cm}\\
	&M = M(\alpha,X,U), & &M' = M(\beta,Y,V).
	\end{align*}
	Consider the parameters
	\begin{gather*}
	C:= m \cdot m' \cdot (1-A \cdot L^{-\kappa}).\\
	C':=M \cdot M' \cdot (1+A \cdot L^{-\kappa}),
	\end{gather*}
	The compactness of $K \subseteq \mathcal{C}_g^*$, $\mathcal{K} \subseteq \mathcal{T}_g$, and $\pmf$, and the assumption that $\beta$ is filling ensure that
	\begin{align}
	\hspace{+4.1cm} &m \asymp_{\mathcal{K},K} 1, & &m' \asymp_{\mathcal{K},\beta} 1, \hspace{+4.9cm} \label{eq:m3}\\
	&M \asymp_{\mathcal{K},K} 1, & &M' \asymp_{\mathcal{K},\beta} 1. \label{eq:m4}
	\end{align}
	As a consequence of (\ref{eq:m3}) and (\ref{eq:m4}) we deduce
	\begin{equation}
	\label{eq:m5}
	\frac{1}{C^{6g-6}} - \frac{1}{(C')^{6g-6}} = (C' - C) \cdot \frac{\sum_{i=0}^{6g-5} (C')^{6g-5-i} \cdot C^i}{C^{6g-6} \cdot (C')^{6g-6}} \preceq_{\mathcal{K},K,\beta} C'- C.
	\end{equation}
	A direct computation shows that 
	\begin{equation}
	\label{eq:m6}
	C' - C = \left( M \cdot M' - m \cdot m' \right) + \left(M \cdot M' + m \cdot m'\right) \cdot A \cdot L^{-\kappa}.
	\end{equation}
	As a consequence of (\ref{eq:m3}) and (\ref{eq:m4}) we deduce
	\begin{equation}
	\label{eq:m7}
	\left(M \cdot M' + m \cdot m'\right) \cdot A \cdot L^{-\kappa} \preceq_{\mathcal{K},K,\beta} A \cdot L^{-\kappa}.
	\end{equation}
	A direct computation using (\ref{eq:m3}) and (\ref{eq:m4}) implies
	\begin{equation}
	\label{eq:m8}
	M \cdot M' - m \cdot m' = (M-m) \cdot M' + m \cdot (M'-m') \preceq_{\mathcal{K},K,\beta} (M-m) + (M' - m').
	\end{equation}
	Proposition \ref{prop:bound_1} guarantees
	\begin{gather*}
	M - m \preceq_{\mathcal{K},K,F} \mathrm{diam}_F(U),\\
	M' - m' \preceq_{\mathcal{K},\beta,F} \mathrm{diam}_F(V).
	\end{gather*}
	Putting together (\ref{eq:m5}), (\ref{eq:m6}), (\ref{eq:m7}), and (\ref{eq:m8}) we deduce
	\begin{equation}
	\label{eq:m9}
	\frac{1}{C^{6g-6}} - \frac{1}{(C')^{6g-6}} \preceq_{\mathcal{K},K,\beta,F} \mathrm{diam}_F(U) + \mathrm{diam}_F(V) + A \cdot L^{-\kappa}.
	\end{equation}
	
	Recall the definition of the constant $\Lambda_g > 0$ in (\ref{eq:hmc}). Notice that
	\begin{equation}
	\label{eq:m10}
	\nu_X(U) \leq \Lambda_g, \quad \nu_Y(V) \leq \Lambda_g.
	\end{equation}
	Putting together (\ref{eq:ma}), (\ref{eq:mb}), (\ref{eq:m9}), and (\ref{eq:m10}) we conclude
	\[
	S - S' \preceq_{\mathcal{K},K,\beta,F} \nu_X(U) \cdot \nu_Y(V) \cdot\left(\mathrm{diam}_F(U) + \mathrm{diam}_F(V)\right) + A \cdot L^{-\kappa}. \qedhere
	\]
\end{proof}

We are now ready to prove the main estimate for the local counting function $f(\alpha,\beta,X,Y,U,V,L)$. We will later deduce Theorem \ref{theo:main} as a direct consequence of this estimate.

\begin{theorem}
	\label{theo:count_loc}
	There exists $\kappa_4 = \kappa_4(g) > 0$ such that for every $K \subseteq \mathcal{C}_g^*$ compact, every $\alpha \in K$, every filling closed curve $\beta$ on $S_g$, every $\mathcal{K} \subseteq \mathcal{T}_g$ compact, every $X,Y \in \mathcal{K}$, every set of Dehn-Thurston coordinates $F \colon \mf \to \Sigma^{3g-3}$, every non-empty $F$-simplices $U,V \subseteq \pmf$, and every $L > 0$,
	\begin{gather*}
	f(\alpha,\beta,X,Y,U,V,L) \\= \frac{\nu_\alpha(U) \cdot \nu_\beta(V)}{b_g} \cdot L^{6g-6} + O_{\mathcal{K},K,\beta,F}\left(\nu_X(U) \cdot \nu_Y(V) \cdot (\mathrm{diam}_F(U)+ \mathrm{diam}_F(V)) \cdot L^{6g-6} + L^{6g-6-\kappa_4} \right). 
	\end{gather*}
\end{theorem}

\begin{proof}
	Let $\kappa_4 = \kappa_4(g) > 0$ be as in Proposition \ref{prop:total_DT}. Fix $K \subseteq \mathcal{C}_g^*$ compact, a filling closed curve $\beta$ on $S_g$, and $\mathcal{K} \subseteq \mathcal{T}_g$ compact. Let $A = A(\mathcal{K},K,\beta) > 0$ and $L_2 =  L_2(\mathcal{K},K,\beta) > 0$ be as in Proposition \ref{prop:total_DT}, $L_3 = L_3(A,\kappa_4) > 0$ be as in Proposition \ref{prop:bound_2}, and $L_4 = L_4(\mathcal{K},K,\beta) := \max\{L_2,L_3\} > 0$. Fix $\alpha \in K$, $X,Y \in \mathcal{T}_g$, a set of Dehn-Thurston coordinates $F \colon \mf \to \Sigma^{3g-3}$, a pair of non-empty $F$-simplices $U,V \subseteq \pmf$, and $L > L_4$. Consider the parameters
	\begin{gather*}
	C:= m(\alpha,X,U) \cdot m(\beta,Y,V) \cdot (1-A \cdot L^{-\kappa_4}),\\
	C':= M(\alpha,X,U) \cdot M(\beta,Y,V) \cdot (1+A \cdot L^{-\kappa_4}).
	\end{gather*}
	Proposition \ref{prop:total_DT} guarantees
	\begin{gather}
	f(\alpha,\beta,X,Y,U,V,L) - \frac{\nu_X(U) \cdot \nu_Y(V)}{C^{6g-6} \cdot b_g} \cdot L^{6g-6} \preceq_{\mathcal{K},K,\beta,F} L^{6g-6-\kappa_4}, \label{eq:l1}\\
	\frac{\nu_X(U) \cdot \nu_Y(V)}{(C')^{6g-6} \cdot b_g} \cdot L^{6g-6}  - f(\alpha,\beta,X,Y,U,V,L) \preceq_{\mathcal{K},K,\beta,F} L^{6g-6-\kappa_4}. \label{eq:l2}
	\end{gather}
	Consider the parameters
	\begin{gather*}
	S := \nu_X(U) \cdot \nu_Y(V) \cdot C^{-(6g-6)},\\
	S' := \nu_X(U) \cdot \nu_Y(V) \cdot (C')^{-(6g-6)}.
	\end{gather*}
	Proposition \ref{prop:bound_2} guarantees
	\begin{gather}
	S' \leq \nu_\alpha(U) \cdot \nu_\beta(V) \leq S, \label{eq:l3}\\
	S - S' \preceq_{\mathcal{K},K,\beta,F} \nu_X(U) \cdot \nu_Y(V) \cdot \left(\mathrm{diam}_F(U) + \mathrm{diam}_F(V)\right) + A \cdot L^{-\kappa_4}. \label{eq:l4}
	\end{gather}
	Putting together (\ref{eq:l1}), (\ref{eq:l2}), (\ref{eq:l3}), and (\ref{eq:l4}) we conclude
	\begin{gather*}
	f(\alpha,\beta,X,Y,U,V,L) \\= \frac{\nu_\alpha(U) \cdot \nu_\beta(V)}{b_g} \cdot L^{6g-6} + O_{\mathcal{K},K,\beta,F}\left(\nu_X(U) \cdot \nu_Y(V) \cdot \left(\mathrm{diam}_F(U)+ \mathrm{diam}_F(V)\right) \cdot L^{6g-6} + L^{6g-6-\kappa_4} \right). 
	\end{gather*}
	The same estimate holds for every $L > 0$ by increasing the implicit constant.
\end{proof}

\subsection*{From local to global counting estimates.} We are now ready to prove Theorem \ref{theo:main}, the main result of this paper. Recall the definitions of the counting functions $c(\alpha,\beta,L)$ and $f(\alpha,\beta,L)$ in (\ref{eq:c}) and (\ref{eq:f}). Recall that for every filling geodesic current $\alpha \in \mathcal{C}_g^*$ we consider the positive constant
\[
c(\alpha) := \nu_{\mathrm{Thu}}\left(\{\lambda \in \mf \ | \ i(\alpha,\lambda) \leq 1\}\right).
\]

We will actually prove the following stronger version of Theorem \ref{theo:main} with a more explicit description of the leading term and a more precise control of the implicit constant.

\begin{theorem}
	\label{theo:main_strong}
	There exists a constant $\kappa = \kappa(g) > 0$ such that for every compact subset $K \subseteq \mathcal{C}_g^*$, every filling geodesic current $\alpha \in K$, every filling closed curve $\beta$ on $S_g$, and every $L > 0$,
	\[
	c(\alpha,\beta,L) = \frac{c(\alpha) \cdot c(\beta)}{|\mathrm{Stab}(\beta)| \cdot b_g} \cdot L^{6g-6} + O_{K,\beta}\left(L^{6g-6-\kappa}\right).
	\]
\end{theorem}

\begin{proof}
	Let $\kappa_4 = \kappa_4(g) > 0$ be as in Theorem \ref{theo:count_loc}. For the rest of this proof we fix an auxiliary pair of points $X,Y \in \mathcal{T}_g$ and an auxiliary set of Dehn-Thurston coordinates $F \colon \mf \to \Sigma^{3g-3}$ depending only on $g$. Let $K \subseteq \mathcal{C}_g^*$ compact. Fix $\alpha \in K$, a filling closed curve $\beta$ on $S_g$, and $L > 1$. Let $\delta = \delta(L) \in (0,1)$ to be fixed later. Consider a partition $\smash{\{U_i\}_{i=1}^k}$ of $\pmf$ into $k \preceq_g \smash{\delta^{-(6g-7)}}$ $F$-simplices of diameter $\mathrm{diam}_F(U_i) \leq \delta$. Theorem \ref{theo:count_loc} guarantees that, for every $i,j \in \{1,\dots,k\}$, 
	\begin{gather}
	f(\alpha,\beta,X,Y,U_i,U_j,L) \label{eq:g1} \\= \frac{\nu_\alpha(U_i) \cdot \nu_\beta(U_j)}{b_g} \cdot L^{6g-6} + O_{K,\beta}\left(\nu_X(U_i) \cdot \nu_Y(U_j) \cdot \delta \cdot L^{6g-6} + L^{6g-6-\kappa_4} \right). \nonumber
	\end{gather}
	Recall the definition of the constant $\Lambda_g > 0$ in (\ref{eq:hmc}). Notice that
	\begin{align}
	\hspace{+3.2cm} &\sum_{i=1}^k \nu_\alpha(U_i)  = c(\alpha), & &\sum_{j=1}^k \nu_\beta(U_j)  = c(\beta), \hspace{+4.2cm} \label{eq:g3}\\
	&\sum_{i=1}^k \nu_X(U_i)  = \Lambda_g, & &\sum_{j=1}^k \nu_Y(U_j)  = \Lambda_g. \label{eq:g4}
	\end{align}
	Adding (\ref{eq:g1}) over all $i,j \in \{1,\dots,k\}$, using (\ref{eq:g3}) and (\ref{eq:g4}), and recalling $k \preceq_g \delta^{-(6g-7)}$, we deduce
	\begin{equation}
	\label{eq:g5}
	f(\alpha,\beta,L) = \frac{c(\alpha) \cdot c(\beta)}{b_g} \cdot L^{6g-6} + O_{K,\beta}\left(\delta \cdot L^{6g-6} + \delta^{-(12g-14)} \cdot L^{6g-6-\kappa_4}\right).
	\end{equation}
	Let $\delta = L^{-\eta}$ with $\eta = \eta(g) > 0$ small enough so that $(12g-14) \eta < \kappa_4$. Consider $\kappa = \kappa(g) := \min\{\eta,\kappa_4 -(12g-14) \eta\} > 0$. It follows from (\ref{eq:g5}) that 
	\[
	f(\alpha,\beta,L) = \frac{c(\alpha) \cdot c(\beta)}{b_g} \cdot L^{6g-6} + O_{K,\beta}\left(L^{6g-6-\kappa}\right). 
	\]	
	This estimate together with the identity in (\ref{eq:count_cf}) imply that, for every $L > 1$,
	\[
	c(\alpha,\beta,L) = \frac{c(\alpha) \cdot c(\beta)}{|\mathrm{Stab}(\beta)| \cdot b_g} \cdot L^{6g-6} + O_{K,\beta}\left(L^{6g-6-\kappa}\right). 
	\]	
	The same estimate holds for every $L > 0$ by increasing the implicit constant.
\end{proof}

\subsection*{Counting filling closed multi-curves on surfaces.}  We finish this section by stating a stronger version of Theorem \ref{theo:main_strong} that applies to counting functions of non-connected filling closed multi-curves. 

A closed multi-curve $\gamma := (\gamma_i)_{i=1}^k$ on $S_g$ is an ordered tuple of closed curves on $S_g$. The mapping class group $\mcg$ acts naturally on multi-curves of $S_g$ by acting on each of their individual components. Given a geodesic current $\alpha \in \mathcal{C}_g$ and a closed multi-curve $\gamma := (\gamma_i)_{i=1}^k$ on $S_g$ consider the vector $i(\alpha,\gamma) := (i(\alpha, \gamma_i))_{i=1}^k \in \smash{(\mathbf{R}_{\geq 0})^k}$. Let $\alpha \in \mathcal{C}_g^*$ be a filling geodesic current on $S_g$ and $\beta := (\beta_i)_{i=1}^k$ be a closed multi-curve on $S_g$. Consider a homogeneous, Lipschitz function $F \colon \smash{(\mathbf{R}_{\geq 0})^k} \to \mathbf{R}_{\geq 0}$ that is positive away from the origin. Relevant examples include
\begin{gather*}
F(x_1,\dots,x_k) := x_1 + \dots + x_k,\\
F(x_1,\dots,x_k) := \max_{i=1,\dots,k} x_i.
\end{gather*}
 For every $L > 0$ consider the counting function 
\[
c(\alpha,\beta,F,L) := \#\{\gamma \in \mcg \cdot \beta \ | \ F(i(\alpha,\gamma)) \leq L\}.
\]

 A closed multi-curve $\beta := (\beta_i)_{i=1}^k$ on $S_g$ is said to be filling if every homotopically non-trivial closed curve on $S_g$ intersects at least one component of $\beta$.  If $\beta$ is a filling closed multi-curve on $S_g$ then its stabilizer $\mathrm{Stab}(\beta) \subseteq \mcg$ is finite. The tracking method introduced in the proof of Theorem \ref{theo:main} can also be used to prove a quantitative estimate with a power saving error term for the counting function $c(\alpha,\beta,F,L)$ under the assumption that the closed multi-curve $\beta$ is filling.

To give a precise statement of this estimate let us first introduced some notation. Let $k \in \mathbf{N}$ and $F \colon \smash{(\mathbf{R}_{\geq 0})^k} \to \mathbf{R}_{\geq 0}$ be a homogeneous, Lipschitz function. Denote by $\mathrm{Lip}(F) > 0$ the Lipschitz constant of $F$. Denote by $\mathbf{S}^{k-1} \subseteq \mathbf{R}^k$ the unit sphere centered at the origin with respect to the standard Euclidean metric. Consider the constants
\begin{gather*}
m(F) := \min_{x \in \mathbf{S}^{k-1} \cap \smash{(\mathbf{R}_{\geq 0})^k}} F(x),\\
M(F) := \max_{x \in \mathbf{S}^{k-1} \cap \smash{(\mathbf{R}_{\geq 0})^k}} F(x).
\end{gather*}
Notice that the condition $m(F) > 0$ is equivalent to $F \colon \smash{(\mathbf{R}_{\geq 0})^k} \to \mathbf{R}_{\geq 0}$ being positive away from the origin. Given a filling close multi-curve $\beta := (\beta_i)_{i=1}^k$ on $S_g$ consider the positive constant
\[
c(\beta,F) := \nu_{\mathrm{Thu}}\left(\{\lambda \in \mf \ | \ F(i(\lambda,\beta)) \leq 1\}\right)
\]

Minor modifications to the proof of Theorem \ref{theo:main_strong} yield a proof of the following effective estimate for the counting function $c(\alpha,\beta,F,L)$ under the assumption that the closed multi-curve $\beta$ is filling.

\begin{theorem}
	\label{theo:main_multi}
	There exists a constant $\kappa = \kappa(g) > 0$ such that for every compact subset $K \subseteq \mathcal{C}_g^*$, every $\alpha \in K$, every filling closed curve $\beta$ on $S_g$, every $a,b,c > 0$, every homogeneous, Lipschitz function $F \colon \smash{(\mathbf{R}_{\geq 0})^k} \to \mathbf{R}_{\geq 0}$  such that $\mathrm{Lip}(F) < a$, $m(F) > b$, and $M(F) < c$, and every $L > 0$,
	\[
	c(\alpha,\beta,F,L) = \frac{c(\alpha) \cdot c(\beta,F)}{|\mathrm{Stab}(\beta)| \cdot b_g} \cdot L^{6g-6} + O_{K,\beta,a,b,c}\left(L^{6g-6-\kappa}\right).
	\]
\end{theorem}

\section{Counting with respect to Thurston's asymmetric metric}

\subsection*{Outline of this section.} In this section we explain how to adapt the tracking method introduced in the proof of Theorem \ref{theo:main} to prove the effective estimate in Theorem \ref{theo:main_3} for the number of points in a mapping class group orbit of Teichmüller space that lie within a Thurston asymmetric metric ball of given center and large radius. We first state the tracking principle for Thurston's asymmetric metric proved in the prequel \cite{Ara20c}. See Theorem \ref{theo:thurston_track}. We then discuss the regularity of certain functions of interest on the space of singular measured foliations. See Theorems \ref{theo:d_convex} and \ref{theo:d_lip}. We end this section by stating and briefly discussing the proof of a more precise version of Theorem \ref{theo:main_3}. See Theorem \ref{theo:main_3_strong}.

\subsection*{The tracking principle for Thurston's asymmetric metric.} For the rest of this section we fix an integer $g \geq 2$ and a connected, oriented, closed surface $S_g$ of genus $g$. Recall that $\mathcal{T}_g$ denotes the Teichmüller space of marked complex structures on $S_g$, or, equivalently, the Teichmüller space of marked hyperbolic structure on $S_g$. Recall that $d_\mathcal{T}$ denotes the Teichmüller metric on $\mathcal{T}_g$ and that $d_\mathrm{Thu}$ denotes Thurston's asymmetric metric on $\mathcal{T}_g$. Recall that the mapping class group $\mcg$ acts naturally on $\mathcal{T}_g$ by changing the markings. Recall that this action preserves $d_\mathcal{T}$ and $d_\mathrm{Thu}$. 

We now state the tracking principle for Thurston's asymmetric metric proved in the prequel \cite{Ara20c}. Let us first review some relevant notation. Recall that $\qut$ denotes the Teichmüller space of marked unit area quadratic differentials on $S_g$ and that $\pi \colon \qut \to \tt$ denotes the natural projection to $\mathcal{T}_g$. Recall that $\mf$ denotes the space of singular measured foliations on $S_g$ and that $\Re(q), \Im(q) \in \mf$ denote the vertical and horizontal foliations of $q \in \qut$. Recall that $\ell_{\lambda}(X)$ denotes the length of a singular measured foliation $\lambda \in \mf$ with respect to a marked hyperbolic structure $X \in \mathcal{T}_g$.

Recall that $\Delta \subseteq \mathcal{T}_g \times \mathcal{T}_g$ denotes the diagonal of $\mathcal{T}_g \times \mathcal{T}_g$, that $S(X) := \pi^{-1}(X)$ for every $X \in \mathcal{T}_g$, and that $q_s, q_e \colon \mathcal{T}_g \times \mathcal{T}_g - \Delta \to \mathcal{Q}^1\mathcal{T}_g$ denote the maps which to every pair of distinct points $X \neq Y \in \mathcal{T}_g$ assign the quadratic differentials $q_s(X,Y) \in S(X)$ and $q_e(X,Y) \in S(Y)$ corresponding to the cotangent directions at $X$ and $Y$ of the unique Teichmüller geodesic segment from $X$ to $Y$.

Recall that $\mathcal{C}_g$ denotes the space of geodesic currents on $S_g$. Recall that $i(\cdot,\cdot)$ denotes the geometric intersection number pairing on $\mathcal{C}_g$. Recall that $\ell_\beta(X)$ denotes the length of the unique geodesic representative of a free homotopy class  $\beta$ of simple closed curve on $S_g$ with respect to a marked hyperbolic structure $X \in \mathcal{T}_g$. Denote by $\mathcal{S}_g$ be the set of all free homotopy classes of simple closed curves on $S_g$. For every marked hyperbolic structure $X \in \mathcal{T}_g$ denote by $D_X \colon \mf \to \smash{\mathbf{R}}$ the homogeneous function which to every singular measured foliation $\lambda \in \mf$ assigns the value
\begin{equation}
\label{eq:dx}
D_X(\lambda) := \sup_{\beta \in \mathcal{S}_g}\left( \frac{i(\beta,\lambda)}{\ell_\beta(X)}\right).
\end{equation}

Let us highlight an important formula for Thurston's asymmetric metric. By work of Thurston \cite{Thu98}, for every pair of marked hyperbolic structures $X,Y \in \mathcal{T}_g$,
\begin{equation}
\label{eq:Thu}
d_\mathrm{Thu}(X,Y) = \log\left(\sup_{\beta \in \mathcal{S}_g} \frac{\ell_{\beta}(Y)}{\ell_{\beta}(X)}\right).
\end{equation}

In the prequel \cite{Ara20c}, the following tracking principle for Thurston's asymmetric metric is deduced as a consequence of the formula in (\ref{eq:Thu}) and the tracking principle in Theorem \ref{theo:track}. A multiplicative version of this result analogous to Theorem \ref{theo:track_mult} can be proved using similar arguments.

\begin{theorem} \cite[Theorem 9.7]{Ara20c}
	\label{theo:thurston_track}
	There exists a constant $\kappa = \kappa(g) > 0$ such that for every compact subset $\mathcal{K} \subseteq \mathcal{T}_g$ there exists a constant $C  = C(\mathcal{K}) > 0$ with the following property. For every $X,Y \in \mathcal{K}$, there exists an $(X,Y,C,\kappa)$-sparse subset of mapping classes $M = M(X,Y) \subseteq \mcg$ such that for every $\mc \in \mcg \setminus M$, if $r := d_\mathcal{T}(X,\mc.Y)$, $q_s := q_s(X,\mc.Y)$, and $q_e := q_e(X,\mc.Y)$, then
	\[
	d_\mathrm{Thu}(X,\mc.Y) = d_\mathcal{T}(X,\mc.Y) + \log D_X(\Re(q_s)) + \log \ell_{\Im(q_e)}(\mc.Y) + O_\mathcal{K}\left(e^{-\kappa r}\right).
	\]
\end{theorem}

\subsection*{Regularity of the $\boldsymbol{D_X}$ functions.} As in the case of geometric intersection numbers and extremal lengths, the functions $D_X \colon \mf \to \smash{\mathbf{R}}$ introduced in (\ref{eq:dx}) satisfy convenient regularity properties. 

Let $\tau$ be a maximal train track on $S_g$. Recall that the cone $U(\tau) \subseteq \mf$ of singular measured foliations on $S_g$ carried by $\tau$ can be identified with the cone $V(\tau) \subseteq \smash{(\mathbf{R}_{\geq0})^{18g-18}}$ of non-negative counting measures on the edges of $\tau$ satisfying the switch conditions through a natural $\mathbf{R}^+$-equivariant bijection $\Phi_\tau \colon U(\tau) \to V(\tau)$ we refer to as the train track chart induced by $\tau$ on $\mf$. Directly from Theorem \ref{theo:curr_lip} we deduce that the functions $D_X \colon \mf \to \smash{\mathbf{R}}$ are convex in train track coordinates.

\begin{theorem}
	\label{theo:d_convex}
	Let $\tau$ be a maximal train track on $S_g$, $\Phi_\tau \colon U(\tau) \to V(\tau)$ be the  train track chart induced by $\tau$ on $\mf$, and $X \in \mathcal{T}_g$. Then, the composition $D_X \circ \Phi_\tau^{-1} \colon V(\tau) \to \mathbf{R}$ is convex.
\end{theorem}

\begin{proof}
	As the composition $D_X \circ \Phi_\tau^{-1} \colon V(\tau) \to \mathbf{R}$ is homogeneous, to show it is convex, it is enough to verify that for every pair of counting measures $u,v \in V(\tau)$,
	\begin{equation}
	\label{eq:ZZ1}
	D_X \circ \Phi_\tau^{-1}(u+v) \leq D_X \circ \Phi_\tau^{-1} (u) + D_X \circ \Phi_\tau^{-1} (v).
	\end{equation}
	Let $u,v \in V(\tau)$ be arbitrary. Theorem \ref{theo:curr_lip} guarantees that, for every $\beta \in \mathcal{S}_g$,
	\[
	i(\beta, \cdot ) \circ \Phi_\tau^{-1}(u+v) \leq i(\beta, \cdot) \circ \Phi_\tau^{-1} (u)  + i(\beta, \cdot) \circ \Phi_\tau^{-1} (v).
	\]
	Dividing this inequality by $\ell_{\beta}(X) > 0$ we deduce that, for every $\beta \in \mathcal{S}_g$,
	\[
	\frac{i(\beta, \cdot )}{\ell_{\beta}(X)} \circ \Phi_\tau^{-1}(u+v) \leq \frac{i(\beta, \cdot)}{\ell_{\beta}(X)} \circ \Phi_\tau^{-1} (u)  + \frac{i(\beta, \cdot)}{\ell_\beta(X)} \circ \Phi_\tau^{-1} (v).
	\]
	Taking supremum over all $\beta \in \mathcal{S}_g$ we conclude (\ref{eq:ZZ1}) holds, i.e.,
	\[
	D_X \circ \Phi_\tau^{-1}(u+v) \leq D_X \circ \Phi_\tau^{-1} (u) + D_X \circ \Phi_\tau^{-1} (v). \qedhere
	\]
\end{proof}

Recall that $\Sigma$ denotes the piecewise linear manifold $\Sigma := \mathbf{R}^2 / \langle-1\rangle$ endowed with the quotient Euclidean metric. Recall that the product $\IT^{3g-3}$ is a piecewise linear manifold which we endow with the $L^2$ product metric. Recall that any set of Dehn-Thurston coodinates provides a homeomorphism $F \colon \mf \to \IT^{3g-3}$ equivariant with respect to the natural $\mathbf{R}^+$ scaling actions on $\mf$ and $\IT^{3g-3}$. The arguments introduced in the proof of Theorem \ref{theo:curr_lip} can also be used to study the functions $D_X \colon \mf \to \smash{\mathbf{R}}$ by applying Theorem \ref{theo:d_convex} in place of Theorem \ref{theo:current_convex}. More concretely, one can show that the functions $D_X \colon \mf \to \smash{\mathbf{R}}$ are Lipschitz in Dehn-Thurston coordinates.

\begin{theorem}
	\label{theo:d_lip}
	Let $F \colon \mf \to \IT^{3g-3}$ be a set of Dehn-Thurston coordinates of $\mf$ and $\mathcal{K} \subseteq \mathcal{T}_g$ be a compact subset of marked hyperbolic structures on $S_g$. Then, there exists a constant $L = L(F,\mathcal{K}) > 0$ such that for every $X \in \mathcal{K}$ the composition $D_X\circ F^{-1} \colon \IT^{3g-3} \to \mathbf{R}$ is $L$-Lipschitz.
\end{theorem}

\subsection*{Counting with respect to Thurston's asymmetric metric.} We are now ready to discuss the main result of this section. Recall that $A_R(X) \subseteq \mathcal{T}_g$ denotes the closed ball of radius $R > 0$ centered at $X \in \mathcal{T}_g$ with respect to Thurston's asymmetric metric. More precisely,
\[
A_R(X) := \{Y \in \mathcal{T}_g \ | \ d_\mathrm{Thu}(X,Y) \leq R\}.
\]
Recall the definition of the constant $b_g > 0$ in (\ref{eq:bg}). Given $X \in \mathcal{T}_g$ denote by $\mathrm{Stab}(X) \subseteq \mcg$ its stabilizer with respect to the $\mcg$ action on $\mathcal{T}_g$ and consider the positive constants
\begin{gather*}
d(X) := \nu_{\mathrm{Thu}}\left(\{\lambda \in \mf \ | \ D_X(\lambda) \leq 1 \} \right),\\
c(X) := \nu_{\mathrm{Thu}}\left(\{\lambda \in \mf \ | \ \ell_{\lambda}(X) \leq 1 \} \right).
\end{gather*}

The tracking method introduced in the proof of Theorem \ref{theo:main} can also be used to prove a quantitative estimate with a power saving error term for the number of points in a mapping class group orbit of Teichmüller space that lie within a Thurston asymmetric metric ball of given center and large radius. Indeed, carefully following the proof of Theorem \ref{theo:main} but using Theorem \ref{theo:thurston_track} in place of Theorem \ref{theo:track} and Theorem \ref{theo:d_lip} in place of Theorem \ref{theo:curr_lip} yields the following more precise version of Theorem \ref{theo:main_3}.

\begin{theorem}
	\label{theo:main_3_strong}
	There exists a constant $\kappa = \kappa(g) > 0$ such that for every compact subset $\mathcal{K} \subseteq \mathcal{T}_g$, every pair of marked hyperbolic structures $X,Y \in \mathcal{K}$, and every $R > 0$,
	\[
	\#\left(\mcg \cdot Y \cap A_R(X)\right) = \frac{d(X) \cdot c(Y)}{|\mathrm{Stab}(Y)| \cdot b_g}\cdot e^{(6g-6)R} + O_{\mathcal{K}}\left(e^{(6g-6-\kappa)R} \right).
	\]
\end{theorem}


\bibliographystyle{amsalpha}


\bibliography{bibliography}

\end{document}